\documentclass[11pt]{amsart}
\usepackage{graphicx}
\usepackage{graphics}
\usepackage{amsmath}
\usepackage{amscd}
\usepackage{latexsym}
\usepackage[all]{xy}
\begin{document}
\textwidth 5.5in
\textheight 8.3in
\evensidemargin .75in
\oddsidemargin.75in

\newtheorem{exm}{Examples}[section]
\newtheorem{lem}{Lemma}[section]
\newtheorem{conj}{Conjecture}[section]
\newtheorem{defi}{Definition}[section]
\newtheorem{thm}{Theorem}[section]
\newtheorem{cor}{Corollary}[section]
\newtheorem{lis}{List}[section]
\newtheorem{prob}{Problem}[section]
\newtheorem{rmk}{Remark}[section]
\newtheorem{que}{Question}[section]
\newtheorem{prop}{Proposition}[section]
\newtheorem{clm}{Claim}[section]
\newcommand{\p}[3]{\Phi_{p,#1}^{#2}(#3)}
\def\Z{\mathbb Z}
\def\R{\mathbb R}
\def\g{\overline{g}}
\def\odots{\reflectbox{\text{$\ddots$}}}
\newcommand{\tg}{\overline{g}}
\def\proof{{\bf Proof. }}
\def\ee{\epsilon_1'}
\def\ef{\epsilon_2'}
\title{The $E_8$-boundings of homology spheres and negative sphere classes in $E(1)$.}
\author{Motoo Tange}
\thanks{The author was partially supported by JSPS KAKENHI Grant Number 24840006}
\subjclass{57R55, 57R65}
\keywords{Definite spin 4-manifold, Brieskorn homology 3-sphere, minimal genus surface, $E_8$-plumbing}
\address{Institute of Mathematics, University of Tsukuba,
 1-1-1 Tennodai, Tsukuba, Ibaraki 305-8571, Japan}
\email{tange@math.tsukuba.ac.jp}
\date{\today}
\maketitle
\begin{abstract}
We define invariants $\frak{ds}$ and $\overline{\frak{ds}}$, which are the maximal and minimal second Betti number divided by $8$ among definite spin boundings of a homology sphere.
The similar invariants $g_8$ and $\overline{g_8}$ are defined by the maximal (or minimal) product sum of $E_8$-form of bounding 4-manifolds. 
We compute these invariants for some homology spheres.
We construct $E_8$-boundings for some of Brieskorn 3-spheres $\Sigma(2,3,12n+5)$ by handle decomposition.
As a by-product of the construction, some negative classes which consist of addition of several fiber classes plus one sectional
class in $E(1)$ are represented by spheres.
\end{abstract}
\section{Introduction}
\subsection{The spin and negative definite bounding.}
It is well-known that any 3-manifold $Y$ is the boundary of a spin 4-manifold $X$.
Furthermore, if we set some conditions of the intersection form of $X$, it becomes unclear whether there exists the bounding with those conditions.
The Rokhlin theorem says that homology sphere $Y$ with the Rokhlin invariant $\mu(Y)=1$ cannot bound any spin 4-manifold with $\sigma(X)\equiv 0\bmod 16$.

Let $X$ be a spin bounding of a homology sphere $Y$.
We can construct a new spin bounding increasing the one positive and negative eigenvalues of the intersection form
by taking connected-sum $X\#S^2\times S^2$.
In this paper we focus on the construction of spin boundings without positive or negative eigenvalues, i.e., $b_2(X)=|\sigma(X)|$.
Such boundings of homology spheres are called {\it negative- (or positive-) definite boundings}.

Ozsv\'ath and Szab\'o in \cite{OS} defined the integer-valued homology cobordism invariant $d(Y)$ 
for any homology sphere.
It is called {\it the correction term} or {\it $d$-invariant}.
By using this homology cobordism invariant they obtained the following:
\begin{thm}[\cite{OS}]
Let $Y$ be an integral homology 3-sphere.
Then any negative-definite bounding $X$ of $Y$ satisfies the inequality
\begin{equation}
\label{os}
\xi^2+\text{rk}(H^2(X;{\Bbb Z}))\le 4d(Y)
\end{equation}
for each characteristic vector $\xi\in H^2(X)$.
\end{thm}
Hence, the non-negativity of $d(Y)$ is a necessary condition to have a negative-definite bounding.

Furthermore, if $Y$ has a spin negative-definite bounding $X$, then the inequality (\ref{os}) implies
\begin{equation}
\label{spinos}
b_2(X)\le 4d(Y).
\end{equation}
Hence, the condition $d(Y)\ge 0$ is necessary condition for 
the integer homology 3-sphere to have a negative-definite bounding

We introduce another condition for spin negative-definite bounding.
Let $\overline{\mu}$ be the Neumann-Siebenmann invariant defined in \cite{Ne}.
In \cite{U1}, Ue shows the following:
\begin{thm}[\cite{U1}]
\label{Umain}
Suppose that a Seifert rational homology 3-sphere $Y$ with spin structure $c$ bounds a negative-definite 4-manifold $X$ with spin structure $c_X$.
Then 
$$b_2(X)\equiv -8\bar{\mu}(Y,c)\mod 16,$$
$$-\frac{8\bar{\mu}(Y,c)}{9}\le b_2(X)\le -8\bar{\mu}(Y,c).$$
\end{thm}
Hence, $\overline{\mu}(Y,c)\le 0$ is necessary condition for a Seifert spin rational homology 3-sphere
to have a spin negative-definite bounding.
On the other hands, the inequality does not guarantee the existence of such bounding $X$.

A topological space $X$ is said to be {\it homologically 1-connected}, if it is connected and $H_1(X)=\{0\}$.
The bounding genus $|\cdot|$ is defined as follows:
Thus, we define the following invariants.
\begin{defi}
Let $Y$ be a homology 3-sphere.
If $Y$ has a definite spin bounding, then we define $\epsilon(Y)$ as follows:
$$\epsilon(Y)=
\begin{cases}1&Y\text{ has a positive-definite spin bounding with }b_2(X)>0\\
-1&Y \text{has a negative-definite spin bounding with }b_2(X)>0\\
0&Y\text{ has a bounding with }b_2(X)=0.
\end{cases}$$
If $Y$ does not have any definite spin bounding, then we define $\epsilon(Y)=\infty$.
Here, the boundings are assumed all homologically 1-connected.
\end{defi}
The invariant $\epsilon$ is well-defined.
In fact, if a homology 3-sphere $Y$ has two boundings $X_1,X_2$ for two among $\{1,-1,0\}$,
then $X=X_1\cup(-X_2)$ is a definite spin closed 4-manifold with $b_2(X)>0$.
Donaldson's diagonalization theorem in \cite{D} does not allow the existence of that $X$.
\begin{defi}
Let $Y$ be a homology 3-sphere.
We define invariants $\frak{ds},\overline{\frak{ds}}$ on homology 3-spheres as follows:
$$\epsilon(Y)=\infty\Leftrightarrow\frak{ds}(Y)=\overline{\frak{ds}}(Y)=\infty.$$
and otherwise,
$$\frak{ds}(Y)=\max\left\{\frac{b_2(X)}{8}|\partial X=Y\text{, }b_2(X)=|\sigma(X)|,\ w_2(X)=0\right\}$$
$$\overline{\frak{ds}}(Y)=\min\left\{\frac{b_2(X)}{8}|\partial X=Y\text{, }b_2(X)=|\sigma(X)|,w_2(X)=0\right\}.$$
We assume the spin boundings are all homologically 1-connected.
\end{defi}
The rank of unimodular definte quadratic form with even type is divisible by $8$.
Even type means that the square for any element is even.
Thus, the values of these invariants are in ${\Bbb N}\cup\{0\}\cup\{\infty\}$.
By the defintion $0\le \overline{\frak{ds}}(Y)\le\frak{ds}(Y)$ holds.
We do not know whether there exists a homology 3-sphere with $\overline{\frak{ds}}(Y)\neq \frak{ds}(Y)$.

It is known that the difference is bounded in the following sense:
\begin{prop}
Let $Y$ be a homology 3-sphere with finite $\epsilon(Y)$ and $\overline{\frak{ds}}(Y)<\frak{ds}(Y)$.
Then $\frak{ds}(Y)-\overline{\frak{ds}}(Y)\le 8(\overline{\frak{ds}}(Y)+1)$ holds.
\end{prop}
\begin{proof}
Let $X_1,X_2$ be two negative-definite spin boundings with $b_2(X_i)=\beta_i$ and $0<\beta_1<\beta_2$.
Then the invariants of the capped closed spin manifold $X=X_2\cup(-X_1)$
are $b_2(X)=\beta_1+\beta_2$ and $\sigma(X)=\beta_2-\beta_1$.
From Furuta's inequality in \cite{F}, we have $\beta_2\le 9\beta_1+8$.
$$\overline{\frak{ds}}(Y)<\frak{ds}(Y)\le 9\overline{\frak{ds}}(Y)+8$$
Consequently, $\frak{ds}(Y)-\overline{\frak{ds}}(Y)\le 8(\overline{\frak{ds}}(Y)+1)$ holds.
\qed
\end{proof}
On the other hands $\frak{ds}$ can be taken arbitrarily large.
The examples below will be computed later.

For positive integer $n$, the Brieskorn homology 3-spheres
$$\Sigma(4n-2,4n-1,8n-3),\ \Sigma(4n-1,4n,8n-1)$$
$$\Sigma(4n-2,4n-1,8n^2-4n+1),\ \Sigma(4n-1,4n,8n^2-1)$$
have $\frak{ds}=n$.
This will be proven later.

Suppose that a homology 3-sphere $Y$ has a bounding $X$ satisfying
$$\partial X=Y,\ Q_X=nE_8,$$
where $n$ is a negative integer, then $nE_8$ is a direct product of $(-n)$-copies of the negative-definite quadratic form with $E_8$-type.
Then, we call the spin bounding $X$ {\it $E_8$-bounding}.
If the bounding is positive-definite, we call the bounding {\it positive $E_8$-bounding},
and if the bounding is negative-definite, the bounding {\it negative $E_8$-bounding}.

\begin{defi}[$E_8$-genera]
Let $Y$ be a homology 3-sphere with finite $\epsilon(Y)$.
If $Y$ has an $E_8$-bounding, then we define the $E_8$-genera as follows:
$$g_8(Y)=\max\{|n||Y=\partial X\text{ and }, w_2(X)=0,Q_X=nE_8\}$$
$$\overline{g_8}(Y)=\min\{|n||Y=\partial X\text{ and }, w_2(X)=0,Q_X=nE_8\},$$

If $Y$ does not have any $E_8$-bounding, then we define $g_8(Y)$ to be
$$g_8(Y)=+\infty.$$
\end{defi}

When a homology 3-sphere $Y$ has finite $\epsilon(Y)$, 
it is not known whether $Y$ has an $E_8$-bounding or not.

We introduce other related invariants.
\begin{defi}
Let $Y$ be a homology 3-sphere.
Then the bounding genus $|Y|$ of $Y$ is defined to be
$$|Y|:=\begin{cases}\min\{n|\partial X=Y,Q_X=nH\}&\mu(Y)=0,\\\infty& \mu(Y)=1,\end{cases}$$
where the bounding 4-manifold $X$ is restricted to homologically 1-connected 4-manifold.
\end{defi}
This invariant is considered as an h-cobordism invariant 
$$|\cdot|:\Theta_{\Bbb Z}^3\to {\Bbb N}\cup\{0,\infty\}.$$

The $\xi$-invariant in \cite{M} is defined as:
$$\xi(Y)=\max\{p-q|p,q\in{\Bbb Z},q>0,p(-E_8)\oplus qH=Q_X\text{ and }w_2(X)=0,\partial X=Y\}.$$
Bohr and Lee's $m$ in \cite{BL} and $\overline{m}$ are defined as:
$$m(Y)=\max\left\{\frac{5}{4}\sigma(X)-b_2(X)|p,q,\in{\Bbb Z},\partial X=Y,\text{ and }w_2(X)=0\right\}$$
$$\overline{m}(Y)=\min\left\{\frac{5}{4}\sigma(X)-b_2(X)|p,q,\in{\Bbb Z},\partial X=Y,\text{ and }w_2(X)=0\right\}$$
Here, the relationship between $m$ and $\xi$ are as follows:
$$m(-Y)/2=\max\left\{\frac{b_2(N)}{8}-q|q\in{\Bbb Z},\partial X=Y,Q_X\cong N\oplus qH,w_2(X)=0\right.\\$$
$$\hfill\left.\text{ and }N:\text{even negative-definite form}\right\}.$$
Thus we have 
$$m(-Y)/2\le \xi(Y)+1,$$
as seen in \cite{M}.
Here the form $H$ is the quadratic form represented by $\begin{pmatrix}0&1\\1&0\end{pmatrix}$.

We state the $\frac{11}{8}$-conjecture by Y. Matsumoto and an equivalent conjecture in terms of $\frak{ds}$ and bounding genus.
\begin{conj}[Y. Matsumoto ($\frac{11}{8}$-conjecture)]
\label{118}
If $X$ is a closed oriented smooth 4-manifold and $Q_X$ is equivalent to $2k(-E_8)\oplus lH$,
then $l\ge 3|k|$ holds.
\end{conj}
\begin{conj}
\label{equiv118}
Suppose that $Y$ is a homology 3-sphere with $\mu(Y)=0$ and $\frak{ds}(Y)< \infty$.
Then the following is satisfied:
$$|Y|\ge \frac{3}{2}\frak{ds}(Y).$$
\end{conj}
\begin{prop}
Conjecture~\ref{118} and ~\ref{equiv118} are equivalent.
\end{prop}
\begin{proof}
Suppose that $\frac{11}{8}$-conjecture holds.
Let $Y$ be a homology 3-sphere with $\mu(Y)=0$ and $\frak{ds}(Y)<\infty$.
Then there exist two bounding 4-manifolds $X_1,X_2$ satisfying $\partial X_1=Y$ and $\partial X_2=-Y$,
where $X_1$ is a definite spin 4-manifold and $Q_{X_2}\cong nH$.
Gluing $X_1$ and $X_2$ along $Y$ we get a closed 4-manifold with $Q_{X_1\cup X_2}\cong mE_8\oplus nH$.
Thus, $n\ge \frac{3|m|}{2}$ holds.
In particular, we may assume $n=|Y|$ and $m=\frak{ds}(Y)$.

Conversely, suppose that Conjecture~\ref{equiv118} holds.
Let $X$ be a closed 4-manifold with $Q_X=2k(-E_8)\oplus lH$.
Then there exists a homology 3-sphere $Y$ cutting the intersection form, i.e., $X=X_1\cup_YX_2$
and $Q_{X_1}\cong 2k(-E_8)$ and $Q_{X_2}\cong lH$ and $\partial X_2=Y$.
Thus $Y$ satisfies $\mu(Y)=0$ and $\frak{ds}(Y)<\infty$.
Hence, $l\ge |Y|\ge \frac{3}{2}\frak{ds}(Y)\ge \frac{3}{2}|2k|=3|k|$.
This implies $\frac{11}8$-conjecture.
\qed\end{proof}
\subsection{Examples of negative-definite spin boundings.}
The aim of this paper is to find negative-definite spin bounding or $E_8$-bounding for some types of 
Brieskorn homology 3-spheres $\Sigma(a_1,a_2,\cdots,a_n)$.
In this section we list the several results below which are proven later.
The $\frak{ds}$-invariant of all the examples are $0\le \frak{ds}<\infty$.

By the Milnor-fiber construction, we get the following:
\begin{thm}
\label{incmilnor}
For any integer $n$, we set $M_n=\Sigma(2,3,6n-1)\#(-\Sigma(2,3,6n-5))$
then $\epsilon(M_n)=-1$ and $g_8(M_n)=1.$
\end{thm}
The minimal resolution of Brieskorn singularities gives definite boundings for the homology 3-spheres.
We will classify all the minimal resolutions of Brieskorn singularities with boundings with $g_8=1$ and $\epsilon=-1$.
\begin{thm}
\label{resolutionofBrieskorn}
If the minimal resolution of Brieskorn singularity gives a bounding with $g_8=1$ and $\epsilon=-1$, then
the homology 3-sphere is one of $\Sigma(2,3,5)$, $\Sigma(3,4,7)$, $\Sigma(2,3,7,11)$, $\Sigma(2,3,7,23)$
or $\Sigma(3,4,7,43)$.
\end{thm}
We give some examples of minimal resolution of the Brieskorn singularity with large $\frak{ds}$:
\begin{thm}
\label{examnu}
For any integer $n$, we have
$$\frak{ds}(\Sigma(4n-2,4n-1,8n-3))=\frak{ds}(\Sigma(4n-1,4n,8n-1))=n$$
$$\frak{ds}(\Sigma(4n-2,4n-1,8n^2-4n+1))=\frak{ds}(\Sigma(4n-1,4n,8n^2-1))=n.$$
\end{thm}
In the last section we will post a question related to $\frac{11}{8}$-conjecture and the bounding genus.

Even if the minimal resolution of a Brieskorn singularity does not give a spin 4-manifold,
in some cases the additional blow-downs of the 4-manifold can give a spin manifold.

Let $(G,a,b,c)$ be a 1-cycled weighted graph $G$ in the left of Figure~\ref{Explumb}.
The labels on two edges on $G$ are given by 3 integers labeled by $a,b$ with $\gcd(a,b)=1$ as drawn in the figure
and the other (unlabeled) edges are labeled by $1$.
The weight on the vertex intersected by the two edges with $a$ and $b$ is $-2c$ and
the other (unweighted) vertices are weighted by $-2$.
Such a graph can give a smooth 4-manifold with a boundary.
The handle diagram of the manifold is drawn in Figure~\ref{Explumb}.
The component weighted by $-2c$ is the $(a,b)$-torus knot.
\begin{thm}
\label{examples}
The quadruple $(G;a,b,c)$ in Table~\ref{compconfi} with $\gcd(a,b)=1$
gives a Brieskorn homology 3-sphere $\Sigma$ with $g_8(\Sigma)=1$ and $\epsilon(\Sigma)=-1$.

In the case of $((1);1,b,c)$, for some non-negative integer $m$ the homology 3-spheres $\Sigma(p,q,r)$ with the pairs $p,q,r$
in Table~\ref{BdofNegDef} have boundings with $g_8=1$ and $\epsilon=-1$.
\begin{table}
$\begin{array}{|c|c|c|c|}\hline
G&a&b&c\\\hline
(1)&3k-2\ell\pm2&-2k+3\ell\mp2&3k^2-4k\ell+3\ell^2\pm2(2k-2\ell)+2\\\hline
(2)&4k-\ell\pm2&-3k+2\ell\mp2&6k^2-3k\ell+\ell^2\pm2(3k-\ell)+2\\\hline
(3)&4k-3\ell\pm2&-3k+4\ell\mp2&6k^2-9k\ell+6\ell^2\pm2(3k-3\ell)+2\\\hline
(4)&5k-2\ell\pm2&-4k+3\ell\mp2&10k^2-8k\ell+3\ell^2\pm2(4k-2\ell)+2\\\hline
(5)&6k-\ell\pm2&-5k+2\ell\mp2&15k^2-5k\ell+\ell^2\pm2(5k-\ell)+2\\\hline
(6)&12k-4\ell\pm3&-10k+6\ell\mp3&60k^2-40k\ell+12\ell^2\pm6(5k-2\ell)+4\\\hline
(6)&12k-4\ell\pm5&-10k+6\ell\mp5&60k^2-40k\ell+12\ell^2\pm10(5k-2\ell)+11\\\hline
(6)&12k-4\ell\pm1&-10k+6\ell&60k^2-40k\ell+12\ell^2\pm10k+1\\\hline
(6)&12k-4\ell\pm3&-10k+6\ell\mp2&60k^2-40k\ell+12\ell^2\pm2(15k-4\ell)+4\\\hline
(7)&14k-2\ell\pm3&-12k+4\ell\mp3&84k^2-24k\ell+4\ell^2\pm6(6k-\ell)+4\\\hline
(7)&14k-2\ell\pm5&-12k+4\ell\mp5&84k^2-24k\ell+4\ell^2\pm10(6k-\ell)+11\\\hline
(7)&14k-2\ell\pm2&-12k+4\ell\mp1&84k^2-24k\ell+4\ell^2\pm12(2k-\ell)+2\\\hline
(7)&14k-2\ell\pm4&-12k+4\ell\mp3&84k^2-24k\ell+4\ell^2\pm6(8k-\ell)+7\\\hline
\end{array}$
\caption{The negative-definite $E_8$-boundings for $(G;a,b,c)$ in {\sc Figure}~\ref{12possible}}
\label{compconfi}
\end{table}

\begin{table}
$\begin{array}{|c|c|c|}\hline
p&q&r\\\hline
10i+7&15i+8&120i^2+148i+45\\\hline
10i+3&15i+2&120i^2+52i+5\\\hline
20i-8&30i-17&480i^2-464i+109\\\hline
20i+8&30i+7&480i^2+304i+45\\\hline
30i-13&45i-27&1080i^2-1116i+281\\\hline
30i-7&45i-18&1080i^2-684i+101\\\hline
30i+7&45i+3&1080i^2+324i+17\\\hline
30i+13&45i+12&1080i^2+756i+125\\\hline
20i+2&30i-7&480i^2-64i-11\\\hline
20i-2&30i-23&480i^2-256i+21\\\hline
10i+7&15i-2&120i^2+68i-365\\\hline
10i+13&15i+7&120i^2+212i+73\\\hline
60i-28&90i-57&4320i^2-4752i+1277\\\hline
60i-8&90i-27&4320i^2-1872i+173\\\hline
60i+8&90i-3&4320i^2+432i-19\\\hline
60i+28&90i+27&4320i^2+3312i+605\\\hline
\end{array}$
\caption{Brieskorn homology 3-spheres from the blow-downs of the minimal resolution of negative-definite plumbings.}
\label{BdofNegDef}
\end{table}
\begin{figure}
\begin{center}
\input{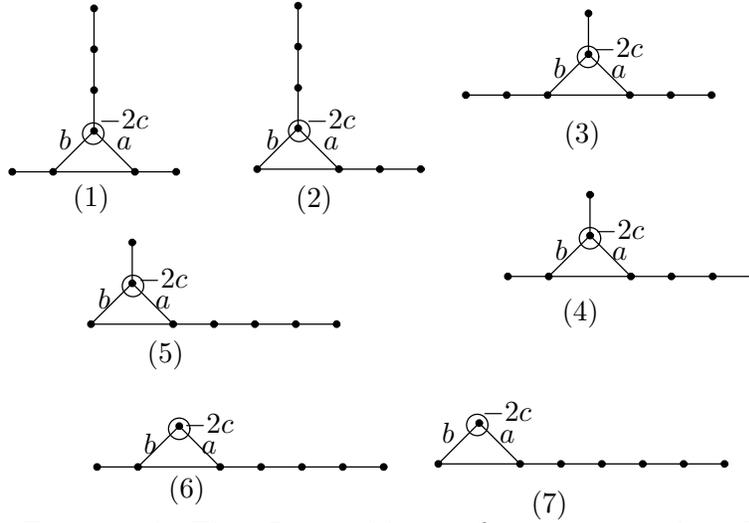}
\caption{The 7 possible configurations with $-E_8$-intersection form.
All the unweighted components are $-2$ and all the labels with unlabeled is $+1$.}
\label{12possible}
\end{center}
\end{figure}
Hence, any Brieskorn 3-sphere above satisfies $g_8(\Sigma)=1$.
\end{thm}
These $E_8$-boundings are constructed by blow-downs of minimal, negative-definite resolutions of Brieskorn singularities.
\subsection{Other examples}
Let $Y_n^-$ denote $\Sigma(2,3,6n-1)$.
Then the Neumann-Siebenmann invariant $\bar{\mu}$ is computed as follows:
\begin{equation}
\label{Ynmu}
\bar{\mu}(Y_n^-)=\begin{cases}-1&n\equiv 1\bmod 2\\0&n\equiv 0\bmod 2.\end{cases}
\end{equation}
As a corollary, if $\frak{ds}(Y_n^-)<\infty$, then $g_8(Y_{2k+1}^-)=1$ and $g_8(Y_{2k}^-)=0$ hold.
In this paper we show the existence of negative-definte spin boundings of $Y_{2k+1}^-$ for some of $k$.
\begin{thm}
\label{mainintro}
For $0\le k\le 12$ or $k=14$, we have $\frak{ds}(Y_{2k+1}^-)<\infty$.
In particular, for these integers $k$ we have $g_8(Y_{2k+1}^-)=1$.
\end{thm}
The boundings cannot be obtained by the minimal resolution or blow-downs of minimal resolutions.
Actually, these boundings can be embedded in $E(1)$ and the complements are Gompf's nuclei $N_{2k+1}$.
\subsection{Embedded spheres in $E(1)$}
Let $E(1)$ be an elliptic fibration diffeomorphic to ${\Bbb C}P^2\#9\overline{{\Bbb C}P^2}$.
According to Li and Li's result in \cite{LL} the spherical realization of the following classes in $E(1)$ are
studied:
\begin{thm}[Li-Li \cite{LL}]
In $H_\ast({\Bbb C}P^2\#n\overline{{\Bbb C}P^2})$ with $1\le n\le 9$ all classes with $0> \xi^2>-(n+7)$,
have minimal genus $0$.
\end{thm}

As a by-product of Theorem~\ref{mainintro} we can obtain the following theorem:
\begin{thm}
\label{maincor2}
Let $f$ and $s$ be the general fiber and a section of elliptic fibration in $E(1)$.
We put $a_k:=k[f]-[s]\in H_2(E(1))$.
For any $0\le k\le 12$ or $k=14$, the class $a_k$ represents an embedded sphere in $E(1)$.
\end{thm}
This intersection number of $a_k$ is $-2k-1$.
Theorem~\ref{maincor2} can be also compared with following Finashin and Mikhalkin's theorem:
\begin{thm}[Finashin-Mikalkin\cite{FM}]
There exists a smooth embedding of $S^2$ into an $E(2)$ with the normal Euler number equal to $n$ for 
any negative even $n\ge -86$.
\end{thm}
In particular, for the general fiber $f$ and a section $s$ in the K3-surface,
the class $k[f]-[s]\in H_2(E(2))$ can be represented by an embedded $S^2$ for $k\le 42$.
We will post a question on the sphere class of $a_k$ in $E(n)$ in the last section.
\subsection*{Acknowledgements}
The results in this article are partially done when I visited in Michigan State University in 2013 spring.
I am grateful for giving useful comments by S. Akbulut and the hospitality of the institute.

This work was supported by JSPS KAKENHI Grant Number 24840006.
\section{Basic properties of invariants $\frak{ds}$ and $g_8$.}
We will prove the basic properties on $\frak{ds}$ and $g_8$.
Let $\Theta_{\Bbb Z}^3$ denote the group of the homology 3-spheres up to h-cobordism.
\begin{thm}
Let $\frak{ds}'$ be one of $\frak{ds}$, $\overline{\frak{ds}}$ and $g_8'$ denote $g_8$, or $\overline{g_8}$.
Then the following properties are satisfied:
\begin{enumerate}
\item The $\frak{ds}'$ and $g_8'$ are h-cobordism invariants i.e., $\frak{ds}':\Theta_{\Bbb Z}^3\to {\Bbb N}\cup\{0,\infty\}$.
\item $\overline{\frak{ds}}(Y)=0$ or $\overline{g_8}(Y)=0$, if and only if $[Y]=0$ in $\Theta_{\Bbb Z}^3$.
\item If $\frak{ds}(Y),g_8(Y)<\infty$, then $\mu(Y)\equiv \frak{ds}'(Y) \equiv g_8'(Y)\equiv 0\bmod 2$
\item If $\epsilon(Y_1)\epsilon(Y_2)=1$, then $\frak{ds}(Y_1)+\frak{ds}(Y_2)\le \frak{ds}(Y_1+Y_2)$.
\item If $\epsilon(Y_1)\epsilon(Y_2)=1$, then $\overline{\frak{ds}}(Y_1+Y_2)\le \overline{\frak{ds}}(Y_1)+\overline{\frak{ds}}(Y_2)$.
\item If $\frak{ds}(Y)=1$, then $g_8(Y)=1$.
\item $\frak{ds}(-Y)=\frak{ds}(Y)$ and $\overline{\frak{ds}}(-Y)=\overline{\frak{ds}}(Y)$.
\item $g_8(-Y)=g_8(Y)$ and $\overline{g_8}(-Y)=\overline{g_8}(Y)$.
\item If $0<\frak{ds}(Y)<\infty$, then $\epsilon(Y)d(Y)<0$ and $\frak{ds}(Y)\le |d(Y)|/2$.
\item If $\frak{ds}'(Y)$ or $g_8'(Y)$ is odd, then $|Y|=\infty$.
\item If $\frak{ds}(Y)$ is even, then we have $\frak{ds}(Y)+1\le |Y|$.
\item If $|Y|=1,2$, then $\frak{ds}(Y)=\infty$.
\item If $\epsilon(Y)\neq \infty$, then $\frak{ds}(Y)-1\le m(-Y)/2-1$.
\item Suppose that $Y$ is a Seifert homology 3-sphere.
If $\frak{ds}(Y)<\infty$, then $\bar{\mu}(Y)\epsilon(Y)>0$ and $\frak{ds}(Y)\le |\bar{\mu}(Y)|$.
\end{enumerate}
\end{thm}
\begin{proof}
(1) Suppose that $Y,Y'$ be h-cobordant homology 3-spheres.
If $\frak{ds}'(Y)<\infty$, then there exists a definite spin bounding $W$ of $Y$ with maximal (or minimal) $b_2$.
Connecting between $Y$ and $Y'$ by the cobordism, we get bounding $W'$ of $Y'$ with a maximal (or minimal) $b_2$.
If $\frak{ds}'(Y)=\infty$ and $\frak{ds}'(Y')$ is finite, then we get a definite spin bounding of $Y$.
This is contradiction.
Thus, if $\frak{ds}'(Y)=\infty$, then $\frak{ds}'(Y')=\infty$

(2) Suppose $Y$ is a homology 3-sphere with $\overline{\frak{ds}}(Y)=0$.
Then $Y$ bounds a homology 4-ball $W$.
Puncturing $W$, we get an h-cobordism between $Y$ and $S^3$.

(3) Suppose that $W$ is any definite spin bounding of $Y$.
Then by the definition of $\mu$ we have $b_2(W)/8\equiv \mu(Y)\bmod 2$.

(4,5) From the properties of maximal and minimal, we have the inequalities by taking the boundary sum of the two definite bounding.

(6) From the property that the definite quadratic form is isomorphic to $\pm E_8$.

(7,8) The definition of $\frak{ds}$ and $g_8$ does not depend on the orientation.

(9) From the inequality~(\ref{spinos}) the inequalities hold.

(10) If $\frak{ds}'(Y)$ or $g'_8(Y)$ is odd, then $\mu(Y)=\frak{ds}(Y)=0(2)$, thus we have $|Y|=\infty$.

(11) If $\frak{ds}(Y)$ is even, then $|Y|<\infty$ holds.
Then we get a closed spin 4-manifold by gluing the two boundings.
The intersection form is isomorphic to $\frak{ds}(Y)\cdot(-E_8)\oplus |Y|\cdot H$.
Furuta's inequality implies $|Y|\ge\frak{ds}(Y)+1$.

(12) If $|Y|=1,2$ and $\frak{ds}(Y)$ is finite, then due to Furuta's inequality $2\ge \frak{ds}(\pm Y)+1$ holds.
Since $\frak{ds}(Y)$ is even, then $\frak{ds}(Y)=0$ namely $[Y]=0$ in $\Theta_{\Bbb Z}^3$.
This contradicts about $|Y|>0$.

(13) The assertion by the definition of $m$ and $\frak{ds}$ is satisfied.

(14) By the result in Theorem~\ref{Umain}, we get the bound of the $\frak{ds}$-invariant.
\end{proof}
\section{The negative $E_8$-bondings}
\subsection{Milnor-fiber construction.}
The Milnor-fiber $M(p,q,r)$ is the 4-manifold defined as the compactification of 
$$\{(x,y,z)\in {\Bbb C}^3|x^p+y^q+z^r=\epsilon\},$$
where $\epsilon$ is some constant.
The boundary is the Brieskorn rational homology 3-spheres $\Sigma(p,q,r)$.
If each two elements in $\{p,q,r\}$ are relatively prime, then the Brieskorn 3-sphere is a homology 3-sphere.
The Milnor-fibers are nice examples of spin bounding.
As mentioned in \cite{GS}, for integers $p,q,r,p',q',r'$ with $p\le p'$, $q\le q'$ and
$r\le r'$, there exists the inclusion $M(p,q,r)\hookrightarrow M(p',q',r')$.
This gives a cobordism between $\Sigma(p,q,r)$ and $\Sigma(p',q',r')$.

{\bf Proof of Theorem~\ref{incmilnor}}.
Here, consider the following natural inclusion:
$$M(2,3,6n-5)\hookrightarrow M(2,3,6n-1).$$
The induced cobordism $X_n$ between $\Sigma(2,3,6n-5)$ and $\Sigma(2,3,6n-1)$ has intersection form $-E_8$.
In fact, it is well-known that $Q_{M(2,3,6n-5)}\cong (n-1)(-E_8)\oplus 2(n-1)H$ and $Q_{M(2,3,6n-1)}\cong n(-E_8)\oplus 2(n-1)H$.
Thus, $Q_{X_n}$ is isomorphic to $-E_8$.

By removing one 3-handle from $X_n$, we get a cobordism $W_n$ from a punctured $\Sigma(2,3,6n-5)$ to 
punctured $\Sigma(2,3,6n-1)$.
The manifolds $Y_n$ is $\partial W_n=M_n$ and
$Q_{W_n}\cong -E_8$.

On the other hands, since $d(M_n)=d(\Sigma(2,3,6n-1)-d(\Sigma(2,3,6n-5)=2-0=2$,
we get $\frak{ds}(M_n)=g_8(M_n)=\epsilon(M_n)=1$.
\qed

\subsection{The minimal resolution.}
Let $W(G)$ be a plumbed 4-manifold associated with a graph $G$, which is a tree weighted by integer.
\begin{defi}
Let $G$ be a connected star-shaped graph as in Figure~\ref{Seifert}.
The `star-shaped' means the graph has at most one $n$-valent vertex with $n\ge 3$.
Let $\{v_0,v_i^j\}$ be the vertices and $m_0$ and $m_i^j$ be the weights of the vertices $v_0$ and $v_i^j$.
That is, the unique vertex $v_0$ is at least 3-valent and the valencies of the other vertices are all $1$ or $2$.
If $G$ satisfies the following properties, we call the graph $G$ is minimal:
\begin{enumerate}
\item The incidence matrix is negative-definite.
\item $m_0\le -1$.
\item $m_i^j\le -2$.
\end{enumerate}
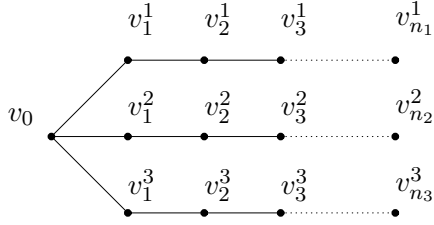
\begin{figure}[htbp]
\begin{center}
\unitlength 0.1in
\begin{picture}( 23.4300, 11.0300)( 12.6000,-20.0300)
%
\special{pn 20}%
\special{sh 1}%
\special{ar 1800 1600 10 10 0  6.28318530717959E+0000}%
\put(17.1000,-15.4000){\makebox(0,0)[rb]{$v_0$}}%
\put(22.0000,-10.7000){\makebox(0,0)[lb]{$v_1^1$}}%
\put(26.0000,-10.7000){\makebox(0,0)[lb]{$v_2^1$}}%
\put(30.0000,-10.7000){\makebox(0,0)[lb]{$v_3^1$}}%
\put(36.0000,-10.7000){\makebox(0,0)[lb]{$v_{n_1}^1$}}%
%
\special{pn 20}%
\special{sh 1}%
\special{ar 2200 1200 10 10 0  6.28318530717959E+0000}%
%
\special{pn 20}%
\special{sh 1}%
\special{ar 2200 2000 10 10 0  6.28318530717959E+0000}%
%
\special{pn 20}%
\special{sh 1}%
\special{ar 2600 2000 10 10 0  6.28318530717959E+0000}%
%
\special{pn 20}%
\special{sh 1}%
\special{ar 3000 2000 10 10 0  6.28318530717959E+0000}%
%
\special{pn 20}%
\special{sh 1}%
\special{ar 3600 2000 10 10 0  6.28318530717959E+0000}%
%
\special{pn 20}%
\special{sh 1}%
\special{ar 3600 1600 10 10 0  6.28318530717959E+0000}%
%
\special{pn 20}%
\special{sh 1}%
\special{ar 3000 1600 10 10 0  6.28318530717959E+0000}%
%
\special{pn 20}%
\special{sh 1}%
\special{ar 2600 1600 10 10 0  6.28318530717959E+0000}%
%
\special{pn 20}%
\special{sh 1}%
\special{ar 2200 1600 10 10 0  6.28318530717959E+0000}%
%
\special{pn 20}%
\special{sh 1}%
\special{ar 3600 1200 10 10 0  6.28318530717959E+0000}%
%
\special{pn 20}%
\special{sh 1}%
\special{ar 3000 1200 10 10 0  6.28318530717959E+0000}%
%
\special{pn 20}%
\special{sh 1}%
\special{ar 2600 1200 10 10 0  6.28318530717959E+0000}%
%
\special{pn 4}%
\special{pa 1800 1600}%
\special{pa 2200 1200}%
\special{fp}%
%
\special{pn 4}%
\special{pa 3000 2000}%
\special{pa 3600 2000}%
\special{dt 0.027}%
%
\special{pn 4}%
\special{pa 3000 1600}%
\special{pa 3600 1600}%
\special{dt 0.027}%
%
\special{pn 4}%
\special{pa 3000 1200}%
\special{pa 3600 1200}%
\special{dt 0.027}%
%
\special{pn 4}%
\special{pa 2200 2000}%
\special{pa 3000 2000}%
\special{fp}%
%
\special{pn 4}%
\special{pa 1800 1600}%
\special{pa 2200 2000}%
\special{fp}%
%
\special{pn 4}%
\special{pa 1800 1600}%
\special{pa 3000 1600}%
\special{fp}%
%
\special{pn 4}%
\special{pa 2200 1200}%
\special{pa 3000 1200}%
\special{fp}%
\put(22.0000,-15.3000){\makebox(0,0)[lb]{$v_1^2$}}%
\put(26.0000,-15.3000){\makebox(0,0)[lb]{$v_2^2$}}%
\put(30.0000,-15.3000){\makebox(0,0)[lb]{$v_3^2$}}%
\put(36.0000,-15.3000){\makebox(0,0)[lb]{$v_{n_2}^2$}}%
\put(22.0000,-19.4000){\makebox(0,0)[lb]{$v_1^3$}}%
\put(26.0000,-19.4000){\makebox(0,0)[lb]{$v_2^3$}}%
\put(30.0000,-19.4000){\makebox(0,0)[lb]{$v_3^3$}}%
\put(36.0000,-19.4000){\makebox(0,0)[lb]{$v_{n_3}^3$}}%
\end{picture}%
\caption{Seifert diagram with three branches.}
\label{Seifert}
\end{center}
\end{figure}
\end{defi}

The minimal graph gives a negative-definite plumbing 4-manifold with a Seifert rational homology 3-sphere boundary.
Furthermore, if all the weights are even, then the plumbing 4-manifold is a spin negative-definite bounding.
We prove the following:
\begin{prop}
If $\Sigma(p,q,r)$ is a Brieskorn homology 3-sphere whose minimal resolution with negative-definite gives an $E_8$-bounding with $b_2=8$.
Then $\Sigma(p,q,r)=\Sigma(2,3,5)$ or $\Sigma(3,4,7)$.
\end{prop}
\begin{proof}
The minimal resolution graph of the Seifert structure we require is rank$=8$, unimodular, negative-definite and even.
Since the graph is even, the weight of the central vertex is $-2$.
The three possible lengths $n_1\le n_2\le n_3$ of branches are $(n_1,n_2,n_3)=(1,2,4),(2,2,3)$, in fact other ones $(1,1,5),(1,3,3)$ cannot be unimodular.

Let us consider the case of $(2,2,3)$ as in Figure~\ref{223},
where $b,c,d,e,f,g,h$ are positive integers.
Let $D=D(b,c,d,e,f,g,h)$ denote the determinant of the resolution graph.
Since the graph gives a homology 3-sphere, $D=1$ holds.
The coefficient of $f$ in $D$ is 
\begin{eqnarray*}
&&4(16bcde-4cde-4bce-4de-4bc+1+c+e)(4gh-1)\\
&\ge &12(4(bc-1)(de-1)+4ce(b(d-1)+d(b-1))+c+e)>0\\
\end{eqnarray*}
Thus $D\ge D(b,c,d,e,1,g,h)$ holds.
By considering the coefficients of $g$ and $h$ in $D(b,c,d,e,1,g,h)$, we have
$D\ge D(b,c,d,e,1,g,h)\ge D(b,c,d,e,1,1,1)$.
If $bc\ge 2$ and $de\ge 2$, then we have
\begin{eqnarray*}
D(b,c,d,e,1,1,1)&=&5+(4bc-5)(4de-5)+32bcd(d-1)+32cde(b-1)+8e+8c\\
&>&1.
\end{eqnarray*}
Thus, this case is not unimodular.
From the symmetry of the graph we may assume $d=e=1$.

Further, if $b\ge 2$, then $D\ge D(b,c,1,1,1,1,1)=28bc-24c-7\ge 32c-7>1$.
Thus we have $b=1$ and due to $D\ge D(1,c,1,1,1,1,1)=4c-7$ we have $c=1,2$.
If we suppose $c=2$, then $1=D\ge D(1,c,1,1,1,1,1)=1$, hence $f=g=h=1$ holds.
This case corresponds to the Brieskorn 3-sphere $\Sigma(3,4,7)$.
If we suppose $c=1$, then $1=D(1,1,1,1,f,g,h)=4(12fgh-9gh-3h-3f)+9\ge 3(f-1)(h-1)+9gh(f-1)+6>1$ holds.
Hence this case does not occur.

In the case of $(1,2,4)$, we use the result in \cite{MY}.
They classified the Brieskorn homology 3-spheres with type $\Sigma(2,q,r)$ and even minimal resolution.
Their theory shows that the homology 3-spheres with rank $8$ among the Brieskorn homology 3-spheres are $\Sigma(2,3,5)$ only .
\begin{figure}[htbp]
\begin{center}
\unitlength 0.1in
\begin{picture}( 16.0000,  9.8000)(  8.0000,-14.0000)
%
\special{pn 20}%
\special{sh 1}%
\special{ar 1200 1000 10 10 0  6.28318530717959E+0000}%
\special{sh 1}%
\special{ar 1000 800 10 10 0  6.28318530717959E+0000}%
\special{sh 1}%
\special{ar 800 600 10 10 0  6.28318530717959E+0000}%
\special{sh 1}%
\special{ar 1000 1200 10 10 0  6.28318530717959E+0000}%
\special{sh 1}%
\special{ar 800 1400 10 10 0  6.28318530717959E+0000}%
\special{sh 1}%
\special{ar 1600 1000 10 10 0  6.28318530717959E+0000}%
\special{sh 1}%
\special{ar 2000 1000 10 10 0  6.28318530717959E+0000}%
\special{sh 1}%
\special{ar 2400 1000 10 10 0  6.28318530717959E+0000}%
%
\special{pn 4}%
\special{pa 2400 1000}%
\special{pa 1200 1000}%
\special{fp}%
\special{pa 1200 1000}%
\special{pa 800 600}%
\special{fp}%
\special{pa 1200 1000}%
\special{pa 800 1400}%
\special{fp}%
\put(12.0000,-8.9000){\makebox(0,0)[lb]{$-2$}}%
\put(15.4000,-9.0000){\makebox(0,0)[lb]{$-2f$}}%
\put(23.2000,-9.0000){\makebox(0,0)[lb]{$-2h$}}%
\put(10.3000,-7.2000){\makebox(0,0)[lb]{$-2b$}}%
\put(8.0000,-5.9000){\makebox(0,0)[lb]{$-2c$}}%
\put(19.2000,-9.0000){\makebox(0,0)[lb]{$-2g$}}%
\put(10.6000,-12.7000){\makebox(0,0)[lb]{$-2d$}}%
\put(8.7000,-15.5000){\makebox(0,0)[lb]{$-2e$}}%
\end{picture}%
\caption{The resolution graph with type $(2,2,3)$.}
\label{223}
\end{center}
\end{figure}
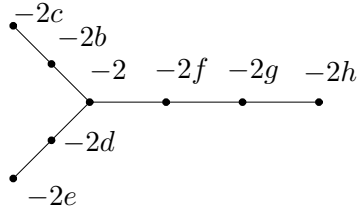
\end{proof}

\begin{prop}
\label{Brieskon4sing}
Let $\Sigma(a_1,a_2,\cdots,a_n)$ be a Brieskorn homology 3-sphere with $n\ge 4$.
Suppose that the minimal resolution graph is even and rank $8$.
Then $n=4$ and the Brieskorn 3-spheres are $\Sigma(2,3,7,11)$, $\Sigma(2,3,7,23)$, or $\Sigma(3,4,7,43)$.
\end{prop}
\begin{proof}
The partitions of $7$, the number of whose parts is more than $4$ has the following $7$ types.
$(4,1,1,1)$, $(3,2,1,1)$, $(2,2,2,1)$, $(3,1,1,1,1)$, $(2,2,1,1,1)$, $(2,1,1,1,1,1,1)$, and $(1,1,1,1,1,1,1)$.

The determinants $D=D(a,b,c,d,e,f,g,h)$ of those matrices with even intersection form are even unless $(2,2,2,1)$.
We may assume the type $(2,2,2,1)$.
\begin{figure}[htbp]
\begin{center}
\unitlength 0.1in
\begin{picture}( 16.0600, 14.1000)( 13.9700,-24.4000)
%
\special{pn 20}%
\special{sh 1}%
\special{ar 2200 2000 10 10 0  6.28318530717959E+0000}%
%
\special{pn 20}%
\special{sh 1}%
\special{ar 2200 2400 10 10 0  6.28318530717959E+0000}%
%
\special{pn 20}%
\special{sh 1}%
\special{ar 1400 2000 10 10 0  6.28318530717959E+0000}%
%
\special{pn 20}%
\special{sh 1}%
\special{ar 1800 2000 10 10 0  6.28318530717959E+0000}%
%
\special{pn 20}%
\special{sh 1}%
\special{ar 2200 1200 10 10 0  6.28318530717959E+0000}%
%
\special{pn 20}%
\special{sh 1}%
\special{ar 2200 1600 10 10 0  6.28318530717959E+0000}%
%
\special{pn 20}%
\special{sh 1}%
\special{ar 3000 2000 10 10 0  6.28318530717959E+0000}%
%
\special{pn 20}%
\special{sh 1}%
\special{ar 2600 2000 10 10 0  6.28318530717959E+0000}%
%
\special{pn 4}%
\special{pa 2200 1200}%
\special{pa 2200 2400}%
\special{fp}%
\special{pa 1400 2000}%
\special{pa 3000 2000}%
\special{fp}%
\put(22.4000,-22.0000){\makebox(0,0)[lb]{$-2a$}}%
\put(26.0000,-22.0000){\makebox(0,0)[lb]{$-2b$}}%
\put(30.0000,-22.0000){\makebox(0,0)[lb]{$-2c$}}%
\put(23.0000,-16.0000){\makebox(0,0)[lb]{$-2d$}}%
\put(23.0000,-12.0000){\makebox(0,0)[lb]{$-2e$}}%
\put(18.0000,-19.3000){\makebox(0,0)[lb]{$-2f$}}%
\put(14.0000,-19.3000){\makebox(0,0)[lb]{$-2g$}}%
\put(22.4000,-24.4000){\makebox(0,0)[lt]{$-2h$}}%
\end{picture}%
\caption{Resolution graph with type $(2,2,2,1)$.}
\label{2221}
\end{center}
\end{figure}
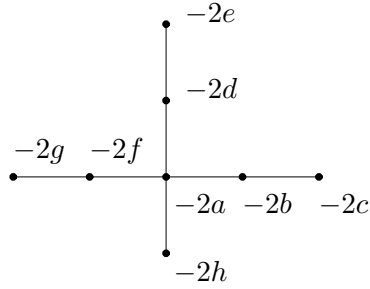
The parameter $a,b,c,d,e,f,g,h$ are all positive numbers.
We put $\tilde{D}=(D-1)/4$.
The unimodular condition is equivalent to $\tilde{D}=0$.
First, the central weight $-2a$ is $-2$ or $-4$ from the even, minimal condition.
We may assume that $b\le d\le f$ from symmetry of the graph.

{\bf [The case of $a=2$.]}
Suppose that $a=2$.
We can take coefficients of $D$ as follows:
$$\tilde{D}=16Nh+16(2h-1)bcdefg+P_1$$
$$N=(2bc-c-2)defg+(2de-e-2)bcfg+(2fg-g-2)bcde,$$
where $P_1$ is a positive integer for any parameter.

If $(b,c)\neq (1,1)$, then $2bc-c-2\ge 0$.
Hence we have $(b,c)=(1,1)$.
Then we have 
\begin{eqnarray*}
&&\tilde{D}(2,1,1,d,e,f,g,h)\\
&=&4((8d-3)e-5)fgh+4((8f-3)g-5)deh+12(h-1)defg+P_2\ge 0,
\end{eqnarray*}
where $P_2$ is a positive integer.

{\bf [The case of $a=1$.]}
Suppose that $a=1$.
$$\tilde{D}=16Nh+16(h-1)bcdefg+P_3$$
$$N=(bc-c-1)defg+(de-e-1)bcfg+(fg-g-1)bcde,$$
where $P_3$ is a positive integer.
We find the case where $N$ is a negative integer.

If $b\ge 2$, we have $bc-c-1,de-e-1,fg-g-1\ge 0$.
Thus $b=1$ holds.
\begin{eqnarray*}
N&=&c(de-e-1)fg+((c-1)fg-(g+1)c)de
\end{eqnarray*}
\begin{lem}
\label{claim0}
Suppose that $b=1$.
If $\tilde{D}=0$, then $d=1$ holds.
\end{lem}
\begin{proof}
Suppose that $d\ge 2$ and $c\ge2$.
We consider the function $\tilde{D}$ with $b=1$.
The partial differential 
\begin{eqnarray*}
\tilde{D}_d&=&48cefgh-16efgh-16cegh-16cdfg-12ceh+P_3\\
&=&16egh((c-1)(f-1)-1)+32cefgh-16cdfg-12ceh+P_3\ge 0.\\
\end{eqnarray*}
Thus we put $\tilde{D}\ge \tilde{D}|_{d=2}:=\tilde{D}_1$.
The partial differential
\begin{eqnarray*}
(\tilde{D}_1)_f&=&80cegh-28egh-32ceg-12cgh+P_4>0.
\end{eqnarray*}
Thus $\tilde{D}\ge (\tilde{D}_1)|_{f=2}=:\tilde{D}_2$.
The partial differential
\begin{eqnarray*}
(\tilde{D}_2)_u&=&128egh-64eg-20gh-20eh+P_5>0.
\end{eqnarray*}
Thus $\tilde{D}\ge \tilde{D}_2|_{c=2}=208egh-112eg-33eh-33gh+P_6>0$.
Here $P_i$ is a positive function.
Therefore, in the case of $d\ge 2$ and $c\ge 2$, we cannot find any solution.

Suppose that $d\ge 2$ and $c=1$.
We consider the function $\tilde{D}$ with $b=c=1$.
Then, by iterating the differential, we get
$\tilde{D}\ge 80egh-48eg-13eh-13gh+P_7$, where $P_7$ is a positive function.
Hence $\tilde{D}> 0$.
Therefore when $b=1$, $d=1$ holds.\qed
\end{proof}
Here we assume $c\le e$.
\begin{lem}
\label{claim1}
Suppose that $b=d=1$.
If $\tilde{D}=0$, then $f\le 2$ must hold.
\end{lem}
\begin{proof}
We prove that when $f\ge 3$ holds, $\tilde{D}$ cannot be $0$.
We deal with $\tilde{D}$ as a function with variables $c,e,f,g,h$.
The polynomial $P_i$ appeared in the context below is a positive function.

(A).
Suppose that $e,c\ge 2$.
Then we have $ec-c-e\ge0$.
We consider the function $\tilde{D}$ with $b=d=1$.
The differential $\tilde{D}_f$ is
\begin{eqnarray*}
\tilde{D}_f&=&32cegh-12egh-12cgh-16ceg+P_8\\
&=&12gh(ce-e-c)+20cegh-16ceg-8gh+P_8.\\
\end{eqnarray*}
In this case $\tilde{D}\ge \tilde{D}|_{f=3}:=\tilde{D}_1$
Since the differential $(\tilde{D}_1)_{h}$ is
\begin{eqnarray*}
(\tilde{D}_1)_{h}&=&80ceg-32g-32cg+11g-8ce+3e+3c-1>0,
\end{eqnarray*}
we have $(\tilde{D}_1)\ge \tilde{D}_1\ge \tilde{D}_1|_{h=1}:=\tilde{D}_{2}$
Since the differential $(\tilde{D}_2)_{c}=32eg-20g-4e+2>0$
$\tilde{D}_3:=\tilde{D}_2|_{c=1}=12eg-12g-2e+1=(e-1)(12g-2)-3>0$
Thus we have $\tilde{D}\ge \tilde{D}_1\ge \tilde{D}_2\ge \tilde{D}_3>0$.

(B).
Suppose that $1=c<e$.
We consider $\tilde{D}$ as the function restricted on $b=c=d=1$.
Then we have $\tilde{D}_f=20egh-12eg-8gh+3g>0$.
Thus $\tilde{D}\ge \tilde{D}|_{f=3}:=\tilde{D}_1$ holds.
Since we have $(\tilde{D}_1)_h=48eg-21g-5e+2>0$, $\tilde{D}\ge (\tilde{D}_1)|_{h=1}=:\tilde{D}_2$
Then $\tilde{D}_2=12eg-12g-2e+1=(12g-2)(e-1)-3>0$.
Thus $\tilde{D}>0$.

(C).
Suppose that $e=c=1$.
Then $\tilde{D}=12fgh-9gh-3h-9fg+2$.
This cannot be $0$ for any integers $f,g,h$.
\qed
\end{proof}

\begin{lem}
\label{claim2}
Suppose that $b=d=1$ and $f=2$.
If $\tilde{D}=0$, then $h=1$ holds.
\end{lem}
\begin{proof}
Suppose that $f=2$ and $h\ge2$.
Since 
$\tilde{D}_h=48ceg-20eg-20cg-8ce+P_{11}>0$,
we have
$$\tilde{D}\ge\tilde{D}|_{h=2}=64ceg-32eg-32cg-12ce+12g+5c+5e-2=:\tilde{D}_1.$$
Since $(\tilde{D}_1)_c=64eg-32g-12e+5>0$, $\tilde{D}_1\ge \tilde{D}_1|_{c=1}=32eg-20g-7e+3>0$.
Therefore $\tilde{D}>0$ holds.
\qed
\end{proof}

\begin{lem}
\label{claim3}
Suppose that $b=d=h=1$, $f=2$ and $c\le e$.
If $\tilde{D}=0$, then $(c,e,g)=(1,3,1)$ or $(1,2,3)$.
\end{lem}
\begin{proof}
If $b=d=h=1$, $f=2$, then
$$\tilde{D}=(12(c-1)(e-1)-7)g+4ce(g-1)+2c+2e-1$$
If $e\ge c\ge 2$, then $\tilde{D}>0$.
If $c=1$, then $\tilde{D}=4eg-7g-2e+1$.
The solution of $4eg-7g-2e+1=0$ is $(e,g)=(2,3),(3,1)$.
\qed
\end{proof}
These cases correspond to $\Sigma(2,3,7,11)$ and $\Sigma(2,3,7,23)$.

\begin{lem}
\label{claim4}
Suppose $b=d=f=1$.
If $\tilde{D}=0$, then $h=2$.
Furthermore if we assume $c\le e\le g$, then we have $(c,e,g)=(1,2,22)$.
\end{lem}
\begin{proof}
Suppose that $b=d=f=1$ and $h\ge 3$.
\begin{eqnarray*}
\tilde{D}_h&=&(16ce-8(c+e)+3)g-8ce+3c+3e-1\\
&\ge& 8ce-5(c+e)+2= 3(ce-1)+5(c-1)(e-1)\ge 0.
\end{eqnarray*}
Then we have $\tilde{D}\ge \tilde{D}|_{h=3}:=\tilde{D}_1$ and
$$\tilde{D}_1=32ceg-20cg-20eg-20ce+8c+8e+8g-3$$

(A).
If $c\ge 2$, then $g\ge e\ge 2$ and
$$(\tilde{D}_1)_c=32eg-20g-20e+8=20(e-1)(g-1)+12(eg-1)\ge 0$$
$$\tilde{D}_1\ge (\tilde{D}_1)_{c=2}=16(e-2)(g-2)+28(eg-2)+5>0$$

(B).
If $c=1$, and $g\ge e\ge 2$, then $\tilde{D}_1=12(e-1)(g-1)-7>0$, then
$\tilde{D}\ge \tilde{D}_1>0$.

(C).
If $c=1$ and $e=1$, then $\tilde{D}=3gh-3h-9g+2\neq 0$.

Suppose that $b=d=f=1$ and $h=1$.
Then $\tilde{D}=-4eg-4cg-4ce+2g+2e+2c-1\neq 0(2)$ holds.
Therefore If $b=d=f=1$, then we have $h=2$.

Suppose that $b=d=f=1$, $h=2$ and $c\ge 2$.
Then 
$$\tilde{D}_c=16eg-12e-12g+5\ge (4e-3)(4g-3)-4\ge 21>0,$$
$\tilde{D}\ge 20eg-19g-19e+8=e(20g-19)-19g+8\ge 21g-30>0$ holds.
Hence, $c=1$ holds.
The integer solution of $\tilde{D}=4eg-7g-7e+3=0$ is $(e,g)=(2,22)$.
\qed
\end{proof}
This case corresponds to $\Sigma(3,4,7,43)$.
From Lemma~\ref{claim0}, \ref{claim1}, \ref{claim2}, \ref{claim3} and \ref{claim4},
we get the proof of Proposition~\ref{Brieskon4sing}
\qed
\end{proof}

{\bf Proof of Theorem~\ref{resolutionofBrieskorn}.}
The rest part in the assertion is the computation of $\frak{ds}$.
From the inequality (\ref{spinos}) by Ozsvath and Szabo, the required assertion follows.
In fact, the $d$-invariants of those Brieskorn homology 3-spheres below are all $2$ by N\'emethi's algorithm in \cite{N},
$$\Sigma(2,3,5),\ \Sigma(3,4,7),\ \Sigma(2,3,7,11),\ \Sigma(2,3,7,23),\ \Sigma(3,4,7,43).$$
Hence, these manifolds are all $g_8=1$.
\qed

In the following, we prove Theorem~\ref{examnu}.

{\bf Proof of Theorem~\ref{examnu}.}
The minimal resolution graphs of $\Sigma(4n-2,4n-1,8n-3)$, $\Sigma(4n-1,4n,8n-1)$,
$\Sigma(4n-2,4n-1,8n^2-4n+1)$, and $\Sigma(4n-1,4n,8n^2-1)$ are {\sc Figure}~\ref{t1} and ~\ref{t2}.
The numbers of the parentheses are the lengths of the branches.
\begin{figure}[htbp]\begin{center}
{\unitlength 0.1in%
\begin{picture}( 46.9000, 18.7000)(  0.0000,-18.7000)%
%
\special{pn 4}%
\special{sh 1}%
\special{ar 800 1200 16 16 0  6.28318530717959E+0000}%
%
\special{pn 4}%
\special{sh 1}%
\special{ar 1800 1600 16 16 0  6.28318530717959E+0000}%
%
\special{pn 4}%
\special{sh 1}%
\special{ar 1400 1600 16 16 0  6.28318530717959E+0000}%
%
\special{pn 4}%
\special{sh 1}%
\special{ar 1000 1600 16 16 0  6.28318530717959E+0000}%
%
\special{pn 4}%
\special{sh 1}%
\special{ar 1800 1200 16 16 0  6.28318530717959E+0000}%
%
\special{pn 4}%
\special{sh 1}%
\special{ar 1200 1200 16 16 0  6.28318530717959E+0000}%
%
\special{pn 4}%
\special{sh 1}%
\special{ar 1600 800 16 16 0  6.28318530717959E+0000}%
%
\special{pn 4}%
\special{sh 1}%
\special{ar 1000 800 16 16 0  6.28318530717959E+0000}%
%
\special{pn 4}%
\special{pa 800 1200}%
\special{pa 1000 800}%
\special{fp}%
%
\special{pn 4}%
\special{pa 800 1200}%
\special{pa 1000 1600}%
\special{fp}%
\put(9.3000,-7.5000){\makebox(0,0)[lb]{$-2$}}%
\put(15.3000,-7.5000){\makebox(0,0)[lb]{$-2$}}%
\put(11.3000,-11.5000){\makebox(0,0)[lb]{$-2$}}%
\put(17.3000,-11.5000){\makebox(0,0)[lb]{$-2$}}%
%
\special{pn 8}%
\special{pa 1200 800}%
\special{pa 1400 800}%
\special{dt 0.045}%
%
\special{pn 8}%
\special{pa 1400 1200}%
\special{pa 1600 1200}%
\special{dt 0.045}%
\put(10.2000,-15.5000){\makebox(0,0)[lb]{$-2n$}}%
\put(14.2000,-15.6000){\makebox(0,0)[lb]{$-2$}}%
\put(17.8000,-15.6000){\makebox(0,0)[lb]{$-2$}}%
%
\special{pn 4}%
\special{sh 1}%
\special{ar 2200 1600 16 16 0  6.28318530717959E+0000}%
%
\special{pn 8}%
\special{pa 1000 1600}%
\special{pa 2200 1600}%
\special{fp}%
\put(21.8000,-15.6000){\makebox(0,0)[lb]{$-2$}}%
%
\special{pn 8}%
\special{pa 1600 800}%
\special{pa 1400 800}%
\special{fp}%
%
\special{pn 8}%
\special{pa 1800 1200}%
\special{pa 1600 1200}%
\special{fp}%
\put(7.3000,-12.0000){\makebox(0,0)[rb]{$-2$}}%
%
\put(24.0000,-12.0000){\makebox(0,0){($4n-2$)}}%
\put(10.0000,-20.0000){\makebox(0,0)[lb]{$\Sigma(4n-2,4n-1,8n-3)$}}%
%
\special{pn 4}%
\special{sh 1}%
\special{ar 3300 1200 16 16 0  6.28318530717959E+0000}%
%
\special{pn 4}%
\special{sh 1}%
\special{ar 3900 1600 16 16 0  6.28318530717959E+0000}%
%
\special{pn 4}%
\special{sh 1}%
\special{ar 3500 1600 16 16 0  6.28318530717959E+0000}%
%
\special{pn 4}%
\special{sh 1}%
\special{ar 4300 1200 16 16 0  6.28318530717959E+0000}%
%
\special{pn 4}%
\special{sh 1}%
\special{ar 3700 1200 16 16 0  6.28318530717959E+0000}%
%
\special{pn 4}%
\special{sh 1}%
\special{ar 4100 800 16 16 0  6.28318530717959E+0000}%
%
\special{pn 4}%
\special{sh 1}%
\special{ar 3500 800 16 16 0  6.28318530717959E+0000}%
%
\special{pn 4}%
\special{pa 3300 1200}%
\special{pa 3500 800}%
\special{fp}%
%
\special{pn 4}%
\special{pa 3300 1200}%
\special{pa 3500 1600}%
\special{fp}%
\put(34.3000,-7.5000){\makebox(0,0)[lb]{$-2$}}%
\put(40.3000,-7.5000){\makebox(0,0)[lb]{$-2$}}%
\put(36.3000,-11.5000){\makebox(0,0)[lb]{$-2$}}%
\put(42.3000,-11.5000){\makebox(0,0)[lb]{$-2$}}%
%
\special{pn 8}%
\special{pa 3700 800}%
\special{pa 3900 800}%
\special{dt 0.045}%
%
\special{pn 8}%
\special{pa 3900 1200}%
\special{pa 4100 1200}%
\special{dt 0.045}%
\put(35.2000,-15.5000){\makebox(0,0)[lb]{$-2n$}}%
\put(39.2000,-15.6000){\makebox(0,0)[lb]{$-2$}}%
%
\special{pn 8}%
\special{pa 4100 800}%
\special{pa 3900 800}%
\special{fp}%
%
\special{pn 8}%
\special{pa 4300 1200}%
\special{pa 4100 1200}%
\special{fp}%
\put(32.3000,-12.0000){\makebox(0,0)[rb]{$-2$}}%
\put(47.0000,-8.0000){\makebox(0,0){($4n-2$)}}%
\put(49.0000,-12.0000){\makebox(0,0){($4n-1$)}}%
\put(35.0000,-20.0000){\makebox(0,0)[lb]{$\Sigma(4n-1,4n,8n-1)$}}%
%
\special{pn 8}%
\special{pa 3460 1600}%
\special{pa 3890 1600}%
\special{fp}%
\put(22.0000,-8.0000){\makebox(0,0){($4n-3$)}}%
%
\special{pn 8}%
\special{pa 1000 800}%
\special{pa 1200 800}%
\special{fp}%
%
\special{pn 8}%
\special{pa 800 1200}%
\special{pa 1400 1200}%
\special{fp}%
%
\special{pn 4}%
\special{pa 3690 800}%
\special{pa 3490 800}%
\special{fp}%
%
\special{pn 4}%
\special{pa 3300 1200}%
\special{pa 3900 1200}%
\special{fp}%
\end{picture}}%
\end{center}\caption{}\label{t1}\end{figure}
\begin{figure}[htbp]\begin{center}
{\unitlength 0.1in%
\begin{picture}( 41.8000, 12.5000)(  5.1000,-18.7000)%
%
\special{pn 4}%
\special{sh 1}%
\special{ar 800 1200 16 16 0  6.28318530717959E+0000}%
%
\special{pn 4}%
\special{sh 1}%
\special{ar 1000 1600 16 16 0  6.28318530717959E+0000}%
%
\special{pn 4}%
\special{sh 1}%
\special{ar 1600 1200 16 16 0  6.28318530717959E+0000}%
%
\special{pn 4}%
\special{sh 1}%
\special{ar 1200 1200 16 16 0  6.28318530717959E+0000}%
%
\special{pn 4}%
\special{sh 1}%
\special{ar 1000 800 16 16 0  6.28318530717959E+0000}%
%
\special{pn 4}%
\special{pa 800 1200}%
\special{pa 1000 800}%
\special{fp}%
%
\special{pn 4}%
\special{pa 800 1200}%
\special{pa 1000 1600}%
\special{fp}%
\put(9.8000,-15.5000){\makebox(0,0)[lb]{$-2$}}%
\put(13.3000,-15.5000){\makebox(0,0)[lb]{$-2$}}%
\put(11.3000,-11.5000){\makebox(0,0)[lb]{$-2n$}}%
\put(15.3000,-11.5000){\makebox(0,0)[lb]{$-2$}}%
\put(7.3000,-12.0000){\makebox(0,0)[rb]{$-2$}}%
\put(23.0000,-16.0000){\makebox(0,0){($4n$)}}%
\put(8.0000,-20.0000){\makebox(0,0)[lb]{$\Sigma(4n-2,4n-1,8n^2-4n+1)$}}%
%
\special{pn 4}%
\special{sh 1}%
\special{ar 2900 1200 16 16 0  6.28318530717959E+0000}%
%
\special{pn 4}%
\special{sh 1}%
\special{ar 3100 1600 16 16 0  6.28318530717959E+0000}%
%
\special{pn 4}%
\special{sh 1}%
\special{ar 3900 1200 16 16 0  6.28318530717959E+0000}%
%
\special{pn 4}%
\special{sh 1}%
\special{ar 3300 1200 16 16 0  6.28318530717959E+0000}%
%
\special{pn 4}%
\special{sh 1}%
\special{ar 3500 800 16 16 0  6.28318530717959E+0000}%
%
\special{pn 4}%
\special{sh 1}%
\special{ar 3100 800 16 16 0  6.28318530717959E+0000}%
%
\special{pn 4}%
\special{pa 2900 1200}%
\special{pa 3100 800}%
\special{fp}%
%
\special{pn 4}%
\special{pa 2900 1200}%
\special{pa 3100 1600}%
\special{fp}%
\put(30.3000,-7.5000){\makebox(0,0)[lb]{$-2n$}}%
\put(34.3000,-7.5000){\makebox(0,0)[lb]{$-2$}}%
\put(32.3000,-11.5000){\makebox(0,0)[lb]{$-2$}}%
\put(38.3000,-11.5000){\makebox(0,0)[lb]{$-2$}}%
%
\special{pn 4}%
\special{pa 3500 1200}%
\special{pa 3700 1200}%
\special{dt 0.045}%
%
\special{pn 4}%
\special{pa 3900 1200}%
\special{pa 3700 1200}%
\special{fp}%
\put(28.3000,-12.0000){\makebox(0,0)[rb]{$-2$}}%
\put(49.0000,-15.9000){\makebox(0,0){($4n-2$)}}%
\put(47.0000,-11.9500){\makebox(0,0){($4n-1$)}}%
\put(31.0000,-20.0000){\makebox(0,0)[lb]{$\Sigma(4n-1,4n,8n^2-1)$}}%
%
\special{pn 4}%
\special{sh 1}%
\special{ar 1400 1600 16 16 0  6.28318530717959E+0000}%
%
\special{pn 4}%
\special{sh 1}%
\special{ar 1800 1600 16 16 0  6.28318530717959E+0000}%
%
\special{pn 4}%
\special{pa 1800 1600}%
\special{pa 1400 1600}%
\special{fp}%
\put(17.3000,-15.5000){\makebox(0,0)[lb]{$-2n$}}%
%
\special{pn 4}%
\special{sh 1}%
\special{ar 1620 800 16 16 0  6.28318530717959E+0000}%
\put(9.5000,-7.5000){\makebox(0,0)[lb]{$-2$}}%
\put(15.5000,-7.5000){\makebox(0,0)[lb]{$-2$}}%
%
\special{pn 8}%
\special{pa 1220 800}%
\special{pa 1420 800}%
\special{dt 0.045}%
%
\special{pn 4}%
\special{pa 1620 800}%
\special{pa 1420 800}%
\special{fp}%
\put(22.2000,-7.9500){\makebox(0,0){($4n-3$)}}%
%
\special{pn 4}%
\special{pa 800 1200}%
\special{pa 1600 1200}%
\special{fp}%
%
\special{pn 4}%
\special{pa 1000 1600}%
\special{pa 1100 1600}%
\special{fp}%
%
\special{pn 4}%
\special{pa 1100 1600}%
\special{pa 1300 1600}%
\special{dt 0.045}%
%
\special{pn 4}%
\special{pa 1300 1600}%
\special{pa 1400 1600}%
\special{fp}%
%
\special{pn 8}%
\special{pa 3100 800}%
\special{pa 3500 800}%
\special{fp}%
\put(30.9000,-15.5000){\makebox(0,0)[lb]{$-2$}}%
\put(34.4000,-15.5000){\makebox(0,0)[lb]{$-2$}}%
%
\special{pn 4}%
\special{sh 1}%
\special{ar 3510 1600 16 16 0  6.28318530717959E+0000}%
%
\special{pn 4}%
\special{sh 1}%
\special{ar 3910 1600 16 16 0  6.28318530717959E+0000}%
%
\special{pn 4}%
\special{pa 3910 1600}%
\special{pa 3510 1600}%
\special{fp}%
\put(38.4000,-15.5000){\makebox(0,0)[lb]{$-2n-2$}}%
%
\special{pn 4}%
\special{pa 3110 1600}%
\special{pa 3210 1600}%
\special{fp}%
%
\special{pn 4}%
\special{pa 3210 1600}%
\special{pa 3410 1600}%
\special{dt 0.045}%
%
\special{pn 4}%
\special{pa 3410 1600}%
\special{pa 3510 1600}%
\special{fp}%
%
\special{pn 8}%
\special{pa 1000 800}%
\special{pa 1200 800}%
\special{fp}%
%
\special{pn 8}%
\special{pa 2900 1200}%
\special{pa 3500 1200}%
\special{fp}%
\end{picture}}%
\end{center}\caption{}\label{t2}\end{figure}
\qed

The intersection forms of these minimal resolution grpahs are not isomorphic to $n(-E_8)$ for $n>1$.
We do not know whether the homology 3-spheres have other boundings with $g_8=n$ and $\epsilon=-1$.

\subsection{Blow-downs of the minimal resolution.}
\label{definiteinvariant}
In general, any minimal resolution is a negative-definite bounding with possibly not even.
But there are some $-1$-spheres in the bounding 4-manifold.
By performing blow-downs of the spheres we can get a smaller bounding.
The new bounding is not a resolution any more.
In this section, we give several $E_8$-boundings with $g_8=1$ and $\epsilon=-1$ by using the blow-down of
the minimal resolutions of Brieskorn homology 3-spheres.
These strategies can be also seen in \cite{NR}.

The blow-down process of a Brieskorn homology 3-sphere is described as follows.
As an example, let us consider a plumbing graph with $3$ singular fibers as in the first diagram in {\sc Figure}~\ref{blow-down}.
By doing the blow-down at the central component, we get the next diagram.
The (unlabeled) edge presents the $+1$-linking between corresponding components.
In the next diagram, doing the further blow-down at the $-1$-framed component, we get the third diagram.
The integers nearby the edge are the linking number between the two components.
In the same way we get the fourth diagram.
Here the $(x+6)$-framed component is the $(2,3)$-torus knot.
Here we deal with the diagram as in the left of {\sc Figure}~\ref{Explumb}.
This diagram stands for the handle diagram in the right of the {\sc Figure}~\ref{Explumb}.
In this paper such a graph is called {\it a configuration} and
any graph obtained by several blow-downs of a Seifert plumbing graph is called {\it a blow-downed configuration}.
The integer associated with any vertex is called a {\it weight} and with any edge is called a {\it label}.
The incidence matrix for the configuration naturally gives the quadratic form for the bounding 4-manifold.

Each step of the blow-down performances is based on the formula in {\sc Figure}~\ref{bdformula}.
Here, if the $x$-framed component in the figure is the $(a,b)$-torus knot, then the next $(x+b^2)$-framed component is the $(a+b,b)$-torus knot.
\begin{figure}[htbp]
\begin{center}
\includegraphics{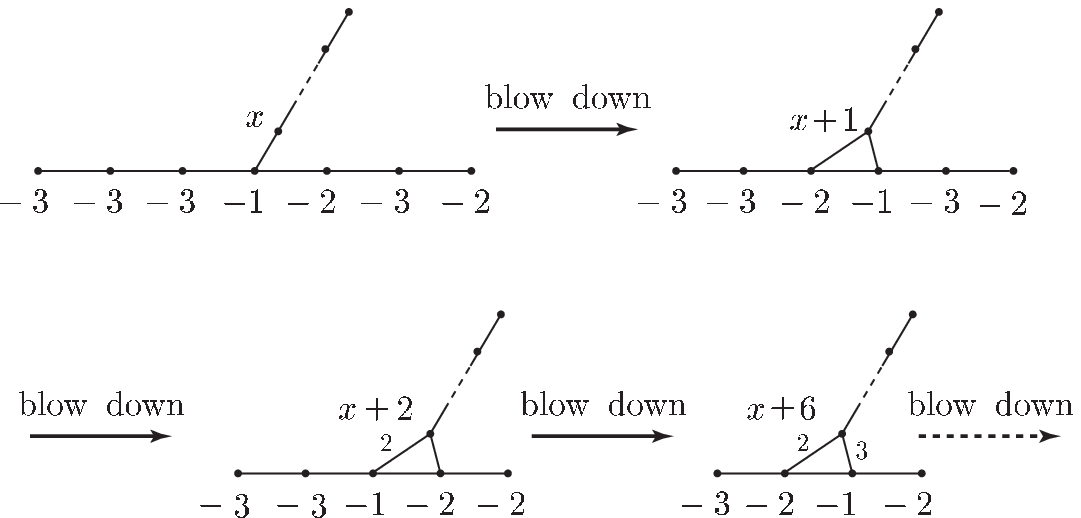}
\caption{Blow-down process.}
\label{blow-down}
\includegraphics{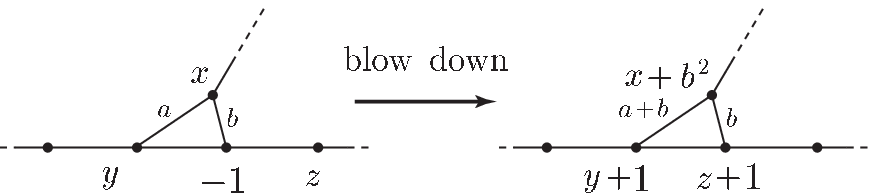}
\caption{A blow-down formula on configurations.}
\label{bdformula}
\end{center}
\end{figure}

Let $G_0$ be a 1-cycled graph with three edges with labels $\{a,b,1\}$.
The vertex intersecting two edges with $a,b$ is $-2c$.
The graph $G$ is the union of $G_0$ and linear edges connecting the three vertices.
See the right of {\sc Figure}~\ref{Explumb}.
We call the graph $G$ a {\it branched triangular configuration}.
\begin{prop}
\label{table1}
Let $G$ be a branched triangular configuration.
The pair $(G;a,b,c)$ with $\gcd(a,b)=1$ in Table~\ref{compconfi} is the blow-downed configurations with type (1) to (7) in {\sc Figure}~\ref{12possible},
whose intersection form is presented by $-E_8$.
\end{prop}
Note that configurations (1) to (7) is not all the blow-downed, branched triangular configurations with $-E_8$ intersection form.

\begin{figure}[htbp]
\begin{center}
\includegraphics{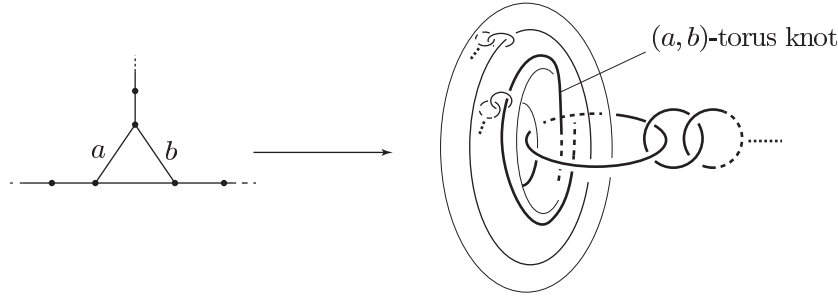}
\caption{The actual handle diagram with branched triangular configuration.}
\label{Explumb}
\end{center}
\end{figure}
\begin{proof}
We consider configurations in {\sc Figure}~\ref{12possible}.
The other branched triangular configuration with rank $8$ cannot be a unimodular form.
Computing the determinants of each configuration in {\sc Figure}~\ref{12possible}, we obtain the equations:
$$(1):3a^2+4ab+3b^2=5c-2;\ (2):3a^2+3ab+2b^2=5c-2$$
$$(3):6a^2+9ab+6b^2=7c-2;\ (4):6a^2+8ab+5b^2=7c-2$$
$$(5):5a^2+5ab+3b^2=7c-2;\ (6):15a^2+20ab+12b^2=16c-1$$
$$(7):12a^2+12ab+7b^2=16c-1$$
The positive integral solutions $\{a,b,c\}$ in these equations give the negative-definite $E_8$-boundings with configurations from (1) to (7).
If $a$ and $b$ are relatively prime, then these pairs $(G;a,b,c)$ are blow-downed configurations by iterating several blow-ups in
accordance with the Euclidean algorithm for relatively prime $(a,b)$ as {\sc Figure}~\ref{intore}.

Suppose that $\{a,b\}$ is a relatively prime solution with $a<b$.
Let $m$ denote the minimum positive number satisfying $b-ma<a$.
We iterate the blow-up process (the inverse of Figure~\ref{bdformula}) $m$-times at the left bottom angle in the triangle as in the first
configuration in {\sc Figure}~\ref{intore}.
Next, exchanging the role of $a$ and $b-ma$, we continue to perform the blow-up at the right bottom angle.
Applying the Euclidean algorighm to this blow-up process in this way, we obtain the star-shaped graph which all labels are $+1$
and all weights are smaller than or equal to $-2$.

In consequence, the pair $(a,b,c)$ in the Table~\ref{compconfi} with relatively prime $a,b$ can give a Brieskorn homology 3-sphere with
$E_8$-bounding with $\epsilon=-1$.
\qed
\end{proof}
\begin{figure}
\begin{center}
\includegraphics{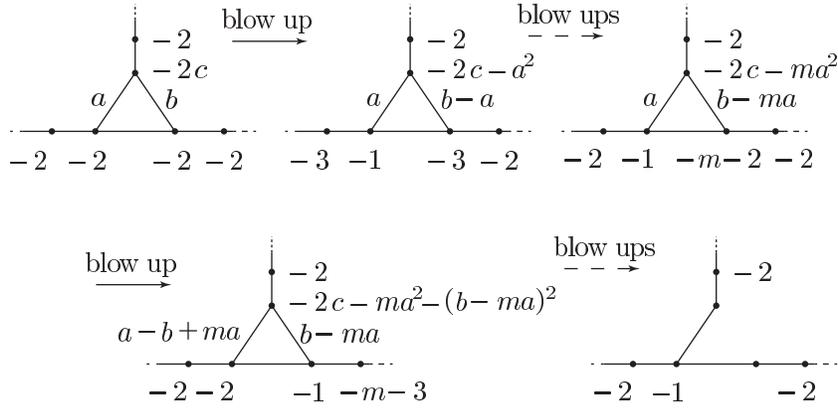}
\caption{Blow-up process by the Euclidean algorithm.}
\label{intore}
\end{center}
\end{figure}
\begin{prop}
\label{list(1)}
Let $G$ be the configuration (1) in Figure~\ref{12possible}.
The integral solutions $a,b,c$ in Table~\ref{compconfi} with $a\le 6$ are Table~\ref{e8bounding1}:\\
\end{prop}
\begin{table}
$\begin{array}{|c|c|c||c|c|c|}\hline
a& b&  c & a & b & c\\\hline
1 & 5i-3 &  15i^2-14i+4 & 1 & 5i & 15i^2+4i+1\\\hline
2 & 10i-7 & 60i^2-68i+21 & 2 & 10i+1 & 60i^2+28i+5\\\hline
3 & 15i-11 & 135i^2-162i+52 & 3 & 15i-8 & 135i^2-108i+25\\\hline
3 & 15i-1 & 135i^2+18i+4 & 3 & 15i+2 & 135i^2+72i+13\\\hline
4 & 10i-5 & 60i^2-28i+9 & 4 & 10i-7 & 60i^2-52i+17\\\hline
5 & 5i-4 & 15i^2-4i+9 & 5 & 5i-1 & 15i^2-14i+12\\\hline
6 & 30i-23 & 540i^2-684i+229 & 6 & 30i-13 & 540i^2-324i+61\\\hline
6 & 30i-5 & 540i^2-36i+13 & 6 & 30i+5 & 540i^2+324i+61\\\hline
\end{array}$
\caption{The pairs $(a,b,c)$ ($i\ge 0$) are blow-downed configurations with (1) with $-E_8$ intersection form and with $a\le 6$.}
\label{e8bounding1}
\end{table}
\begin{proof}
Let us take $a=1$ in the case of (1).
Then for some integer $m$ we have $k=2m+1$, $\ell=3m\pm1+1$ and $b=5m+1\pm1$.
Thus we get 
$$c=15m^2+16m+5\text{, or }15m^2+4m+1$$
from Table~\ref{compconfi}.
In this way we get the expressions of $b,c$ as in Table~\ref{e8bounding1}.

Let $Y$ be one of the Brieskorn homology 3-spheres as above.
According to the formula of $\overline{\mu}$ in \cite{Fk}, we have $\overline{\mu}(Y)=-1$.
From Ue's inequality (Theorem~\ref{Umain}), we get ${\frak{ds}}=1$.
In particular they have $g_8=1$.

Thus, by using Theorem~\ref{Umain} we have $g_8=\overline{g_8}=\frak{ds}=\overline{\frak{ds}}=1$.
These data give Brieskorn homology 3-spheres as in Table~\ref{BdofNegDef}.
\qed
\end{proof}

\subsection{The negative $E_8$ boundings of $\Sigma(2,3,6n-1)$.}
We restrict ourselves to $\Sigma(2,3,6n\pm1)$.
Let denote $Y_{n}^-=\Sigma(2,3,6n-1)$ and $Y_{n}^+=\Sigma(2,3,6n-5)$.
The invariants $\mu$, $\overline{\mu}$ and $d$ for $Y^{\pm}_n$ are as in Table~\ref{236invariant}.
We focus on bounding 4-manifolds of $Y_{2k+1}^-$.
\begin{table}[htbp]
\begin{tabular}{|c|c|c|c|c|}\hline
&$\mu$& $\overline{\mu}$& $d$ & definite bounding \\\hline
$Y_{2k}^+$&$1$&$1$&$0$&$\frak{ds}=\infty$\\\hline
$Y_{2k}^-$&$0$&$0$&$2$&$\frak{ds}=\infty$\\\hline
$Y_{2k+1}^+$&$0$&$0$&$0$&must be $b_2(X)=0$\\\hline
$Y_{2k+1}^-$&$1$&$-1$&$2$&must be $b_2(X)=8$\\\hline
\end{tabular}
\caption{Invariants of $Y_n^{\pm}$.}
\label{236invariant}
\end{table}
The minimal resolution $R_{n}$ for $Y_n^-$ is Figure~\ref{236n-1}.
\begin{figure}[htbp]
\begin{center}
{\unitlength 0.1in%
\begin{picture}( 28.4300,  4.4500)(  7.2000, -8.2500)%
%
\special{pn 4}%
\special{sh 1}%
\special{ar 961 545 16 16 0  6.28318530717959E+0000}%
\special{sh 1}%
\special{ar 1222 545 16 16 0  6.28318530717959E+0000}%
\special{sh 1}%
\special{ar 1481 545 16 16 0  6.28318530717959E+0000}%
\special{sh 1}%
\special{ar 1741 545 16 16 0  6.28318530717959E+0000}%
\special{sh 1}%
\special{ar 2002 545 16 16 0  6.28318530717959E+0000}%
\special{sh 1}%
\special{ar 2261 545 16 16 0  6.28318530717959E+0000}%
%
\special{pn 4}%
\special{sh 1}%
\special{ar 1481 804 16 16 0  6.28318530717959E+0000}%
%
\special{pn 4}%
\special{pa 1481 545}%
\special{pa 1481 804}%
\special{fp}%
%
\special{pn 4}%
\special{sh 1}%
\special{ar 2781 545 16 16 0  6.28318530717959E+0000}%
%
\special{pn 4}%
\special{sh 1}%
\special{ar 2521 545 16 16 0  6.28318530717959E+0000}%
\put(16.1000,-4.4500){\makebox(0,0){$-2$}}%
\put(13.5000,-4.4500){\makebox(0,0){$-2$}}%
\put(10.9100,-4.4500){\makebox(0,0){$-2$}}%
\put(8.3000,-4.4500){\makebox(0,0){$-2$}}%
\put(13.5000,-8.9000){\makebox(0,0){$-2$}}%
\put(18.7100,-4.4500){\makebox(0,0){$-2$}}%
\put(21.3000,-4.4500){\makebox(0,0){$-2$}}%
\put(23.9000,-4.4500){\makebox(0,0){$-2$}}%
\put(26.5000,-4.4500){\makebox(0,0){$-3$}}%
\put(29.1000,-4.4500){\makebox(0,0){$-2$}}%
%
\special{pn 4}%
\special{sh 1}%
\special{ar 3040 540 16 16 0  6.28318530717959E+0000}%
%
\special{pn 8}%
\special{pa 961 545}%
\special{pa 3171 545}%
\special{fp}%
%
\special{pn 8}%
\special{pa 3171 545}%
\special{pa 3444 545}%
\special{dt 0.045}%
%
\special{pn 4}%
\special{sh 1}%
\special{ar 3563 547 16 16 0  6.28318530717959E+0000}%
%
\special{pn 8}%
\special{pa 3562 545}%
\special{pa 3432 545}%
\special{fp}%
\put(34.3000,-4.4500){\makebox(0,0){$-2$}}%
%
\special{pn 8}%
\special{pa 3040 610}%
\special{pa 3080 650}%
\special{fp}%
%
\special{pn 8}%
\special{pa 3080 650}%
\special{pa 3520 650}%
\special{fp}%
%
\special{pn 8}%
\special{pa 3560 610}%
\special{pa 3520 650}%
\special{fp}%
\put(31.3000,-8.0000){\makebox(0,0)[lb]{$n-1$}}%
\end{picture}}%
\caption{The minimal resolution graph of $Y_n^-$.}
\label{236n-1}
\end{center}
\end{figure}
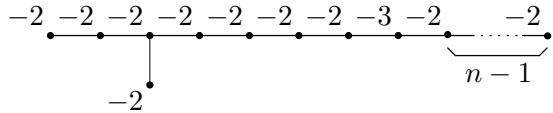
The intersection form of $R_{n}$ is isomorphic to $-E_8\oplus^{n-1}\langle-1\rangle$.
Any square $-1$ class in $R_{n}$ cannot be realized as a sphere,
in other words the following holds:
\begin{prop}
The 4-manifold $R_n$ can be never blow-downed any more.
Namely, the minimal genus of any square $-1$ class in $R_n$ is positive.
\end{prop}
\begin{proof}
Since by replacing any component in {\sc Figure}~\ref{236n-1} with a Legendrian knot as in {\sc Figure}~\ref{stein},
we can get a Stein surface on $R_n$.
On the other hand, any Stein structure does not contain any ($-1$)-sphere.
This means that $R_n$ can be never blow-downed any more.
\qed
\begin{figure}[htbp]
\begin{center}
\includegraphics{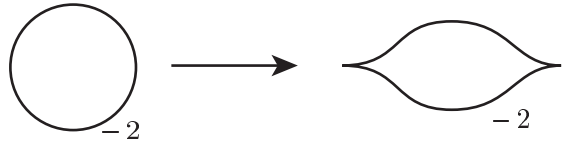}
\caption{A deformation into Stein structure.}
\label{stein}
\end{center}
\end{figure}
\end{proof}

$Y_{2k+1}^-$ has another spin bounding $S_k$ as in Figure~\ref{graph1} with the intersection form isomorphic to $-E_8\oplus \footnotesize{\begin{pmatrix}0&1\\1&0\end{pmatrix}}$.

The direct sum $-E_8\oplus H$ implies the existence of a homology 3-sphere $Y$ separating $-E_8$ and $H$.
The $H$-summand corresponds to a 4-manifold $X$ with $Q_X\cong H$ and $\partial X=Y$.
Can such a homology 3-sphere $Y$ be taken as one satisfying $[Y]=0$ in the homology cobordism group $\Theta_{\Bbb Z}^3$?
We post the following question, which is equivalent to $\frak{ds}(Y_{2k+1}^-)=1$ for any $k$.
\begin{figure}[htbp]
\begin{center}
{\unitlength 0.1in%
\begin{picture}( 25.0700,  5.1100)(  4.9500, -8.1100)%
%
\special{pn 4}%
\special{sh 1}%
\special{ar 757 489 16 16 0  6.28318530717959E+0000}%
\special{sh 1}%
\special{ar 1038 489 16 16 0  6.28318530717959E+0000}%
\special{sh 1}%
\special{ar 1319 489 16 16 0  6.28318530717959E+0000}%
\special{sh 1}%
\special{ar 1319 769 16 16 0  6.28318530717959E+0000}%
\special{sh 1}%
\special{ar 1599 489 16 16 0  6.28318530717959E+0000}%
\special{sh 1}%
\special{ar 1880 489 16 16 0  6.28318530717959E+0000}%
\special{sh 1}%
\special{ar 2160 489 16 16 0  6.28318530717959E+0000}%
\special{sh 1}%
\special{ar 2441 489 16 16 0  6.28318530717959E+0000}%
\special{sh 1}%
\special{ar 2722 489 16 16 0  6.28318530717959E+0000}%
\special{sh 1}%
\special{ar 3002 489 16 16 0  6.28318530717959E+0000}%
\special{sh 1}%
\special{ar 1599 489 16 16 0  6.28318530717959E+0000}%
%
\special{pn 4}%
\special{pa 757 489}%
\special{pa 3002 489}%
\special{fp}%
\special{pa 1319 489}%
\special{pa 1319 769}%
\special{fp}%
\put(7.1500,-4.4700){\makebox(0,0)[rb]{$-2$}}%
\put(9.9600,-4.4700){\makebox(0,0)[rb]{$-2$}}%
\put(12.7600,-4.4700){\makebox(0,0)[rb]{$-2$}}%
\put(15.5700,-4.4700){\makebox(0,0)[rb]{$-2$}}%
\put(18.3800,-4.4700){\makebox(0,0)[rb]{$-2$}}%
\put(21.1800,-4.4700){\makebox(0,0)[rb]{$-2$}}%
\put(23.9900,-4.4700){\makebox(0,0)[rb]{$-2$}}%
\put(26.8000,-4.4700){\makebox(0,0)[rb]{$-2$}}%
\put(12.7600,-8.1100){\makebox(0,0)[rt]{$-2$}}%
\put(29.9000,-4.3000){\makebox(0,0)[rb]{$2k$}}%
\end{picture}}%
\caption{The plumbing graph for $S_k$.}
\label{graph1}
\end{center}
\end{figure}
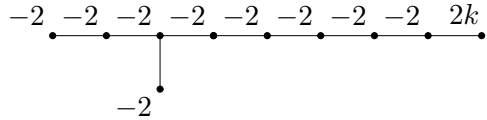
\begin{que}
\label{dividing}
Let $k$ be a positive number.
Can any homology 3-sphere $Y$ in $S_k$ separating the intersection form $Q_{S_k}=-E_8\oplus H$ (i.e., $\partial X=Y$ and $Q_X\cong H$) bound an acyclic 4-manifold?
\end{que}
\subsection{The embedding of $Y_{2k+1}^-$ in $E(1)$.}
Question~\ref{dividing} is unknown, although, we can give several negative $E_8$ boundings for $Y_{2k+1}^-$.

In the case of $n=0$, it is well-known that $Y_1^-=\Sigma(2,3,5)$ is the boundary of the $E_8$-plumbing.
In the case of $n=1$, since $Y_3^+=\Sigma(2,3,13)$ bounds a contractible 4-manifold, we give an h-cobordant
$$Y_3^-\approx Y_3^-\#(-Y_3^+)=M_3.$$
By use of Theorem~\ref{incmilnor}, we can give a negative $E_8$-bounding of $Y_3^-$ with $g_8=1$.
\begin{prop}
\label{main}
For some integer $k$ with $0\le k\le 12,14$, 
$E(1)$ can be decomposed along $Y_{2k+1}^-$ so that $E(1)=W_k\cup_{Y_{2k+1}^-}N_{2k+1}$.
Here $W_k$ is a simply-connected, $E_8$-bounding of $Y_{2k+1}^-$ with $g_8=1$ and $\epsilon=-1$.
\end{prop}
\begin{proof}
We start a well-known decomposition $E(1)=M(2,3,5)\cup N_1$, where $N_1$ is the nuclei, which is defined in \cite{GS}.
{\sc Figure}~\ref{E(1)decom} (Figure~16 in \cite{A}) is the handle diagram for the decomposition.
\begin{figure}[htbp]
\begin{center}
\includegraphics{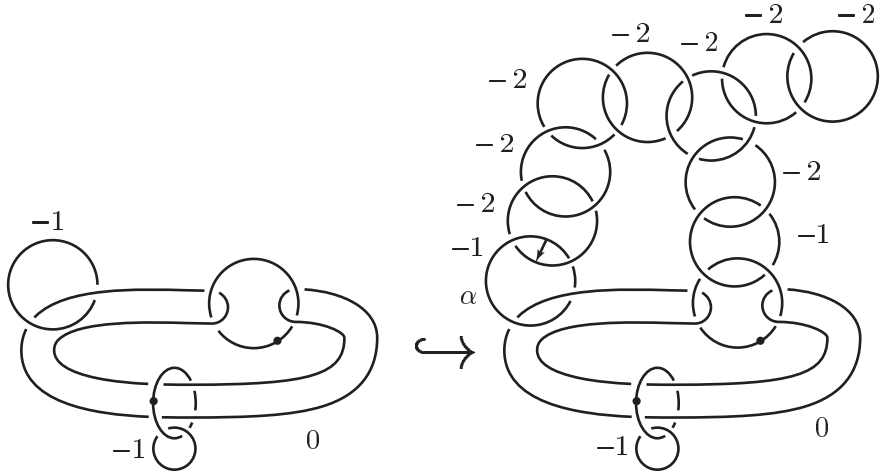}
\caption{Figure16 in \cite{A} and the embedding of $N_1$.}
\label{E(1)decom}
\end{center}
\end{figure}
In the following, we deform the decomposition into other ones via the following 2-handle slide of $\alpha$ in {\sc Figure}~\ref{handleslide}.
The handle slide by a straight band keeps the framing (the left picture in {\sc Figure}~\ref{handleslide}).
On the other hands, the handle slide by a twisting band (the right picture in {\sc Figure}~\ref{handleslide}) decreases the framing by $4$.
Therefore, the framings of $\alpha$ become $-1$ and $-5$ respectively.
\begin{figure}[htbp]
\begin{center}
\includegraphics{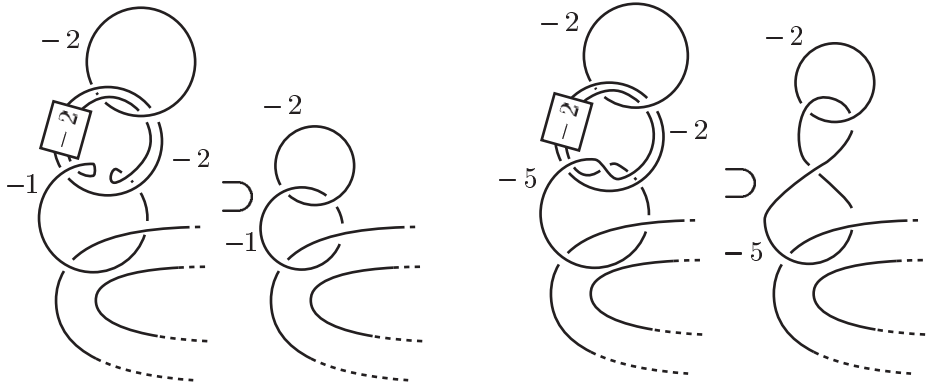}
\caption{The straight handle slide and twisting handle slide.}
\label{handleslide}
\end{center}
\end{figure}
We iterate this process to the linear $7$-component link connecting the $-2$-framed 2-handle except the $-2$-framed 2-handle adjacent to another $-1$-framed 2-handle.
We can realize 2-handle $\alpha$ with the framings of $-1,-5,-9,-13,-17,-21,-25$, and $-29$.
These attaching spheres are all unknots.
The 2-handles with framing $-3,-7,-11,-15,-19$, and $-23$ are obtained by sliding linear sub-$k$-chain ($0\le k\le 5$) and the unknot in the $7$-component link.
For example, {\sc Figure}~\ref{-7realization} realizes a $-7$-framed unknot by sliding $-5$-framed 2-handle to an unconnecting $-2$-framed 2-handle.

This process gives other decomposition $E(1)=W_k\cup_{Y_{2k+1}^-} N_{2k+1}$, where $k$ is $0\le k\le 12$ or $14$.
In fact $W_k$ is a 4-manifold with intersection form $-E_8$ and the boundary is $Y_{2k+1}^-$.
The process above preserves the intersection form of the complement.
As a result, $W_k$ is a simply-connected 4-manifold with intersection form $-E_8$ whose boundary is $Y_{2k+1}^-$ ($0\le k\le 12$ or $14$).
The complement is the nuclei $N_{2k+1}$.
See \cite{GS} for the definition of the nuclei.
\qed
\begin{figure}[htbp]
\begin{center}
\includegraphics{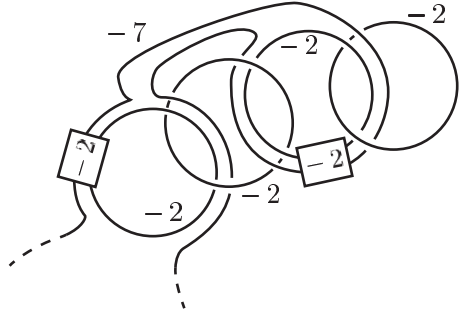}
\caption{A realization of $-7$-framed 2-handle.}
\label{-7realization}
\end{center}
\end{figure}
\end{proof}
{\bf Proof of Theorem~\ref{mainintro}.}
The 4-manifold $W_k$ is a negative $E_8$-bounding of $Y_{2k+1}^-$ for $0\le k\le 12$ or $14$.
Thus, for the integer $k$, we have $g_8(Y_{2k+1}^-)=1$ and $\epsilon(Y_{2k+1}^-)=-1$.
\qed
\medskip\\
There exists an h-cobordism $Y_{2k+1}^+\approx (-M_{2k+1})\#Y_{2k+1}^-$.
The homology 3-sphere $Y_{2k+1}^+\#(-Y_{2k+1}^-)=-M_{2k+1}$ has a `positive' $E_8$-bounding with $g_8=1$ by Theorem~\ref{incmilnor}.
Even if $Y_{2k+1}^-$ has a `negative' $E_8$-bounding with $g_8=1$, 
we do not know whether $Y_{2k+1}^+$ bounds a contractible 4-manifold or not.
In general, what condition for homology spheres $Y_1,Y_2$ with $\epsilon(Y_1)+\epsilon(Y_2)=0$ and $g_8(Y_i)=1$ can cancel out the intersection form $E_8\oplus (-E_8)$ into $\emptyset$?
We pose a more general question in the final section.
\section{The several sphere classes in $E(1)$.}

\begin{thm}
The classes $k[f]-[s]$ ($1\le k\le 13$ or $k=15$) in $H_\ast(E(1))$ are represented by embedded spheres,
where $f$ is the general fiber and $s$ is the section in the elliptic fibration.
\end{thm}
\begin{proof}
The decomposition $W_k\cup_{Y_{2k+1}} N_{2k+1}$ in Theorem~\ref{main} gives $Q_{E(1)}=Q_{W_k}\oplus Q_{N_{2k+1}}\cong -E_8\oplus \begin{pmatrix}-2k-1&1\\1&0\end{pmatrix}$.
Let $\alpha$ denote the same class as the one in Theorem~\ref{main}.
This class is the first class in the $N_{2k+1}$-part.

A diffeomorphism $E(1)\cong {\Bbb C}P^2\#^9\overline{{\Bbb C}P^2}$ induces 
an isomorphism $({\Bbb Z}^{10},Q)\cong ({\Bbb Z}^{10},\langle 1\rangle\oplus9\langle-1\rangle)$
and $\alpha$ is mapped to 
$$-3k\cdot[h]+k\sum_{i=1}^9[e_i]+[e_9]=-k[f]+[s],$$
where $\{[h],[e_i]|1\le i\le 9\}$ is the generator in $H_2({\Bbb C}P^2\#^9\overline{{\Bbb C}P^2})$.
The classes $[f]$ and $[s]$ correspond to the fiber and the section of $E(1)$ respectively.
In the case of $0\le k\le 12,14$, $\alpha$, that is, $-k[f]+[s]$ can be represented as a sphere.
\qed
\end{proof}
\section{Some questions and problems.}
Here we post several questions and problems.
\begin{que}
Let $Y$ be a homology 3-sphere.
\begin{enumerate}
\item When does $Y$ have a definite spin bounding?
\item If $\frak{ds}(Y)<\infty$, then does $Y$ have an $E_8$-bounding?
\item When the equality $m(-Y)/2=\frak{ds}(Y)$ or $\overline{m}(-Y)/2=\overline{\frak{ds}}(Y)$ hold?
\end{enumerate}
\end{que}

\begin{que}
Are there exist any homology 3-spheres $g_8(Y)< \overline{g_8}(Y)$, $\frak{ds}(Y)\neq g_8(Y)$ or $\overline{\frak{ds}}(Y)\neq \overline{g}_8(Y)$?
\end{que}
\begin{que}
Let $Y$ be a Brieskorn homology 3-sphere.
If $4d(Y)=-8\overline{\mu}(Y)>0$, then is $\frak{ds}(Y)=4d(Y)$ true?
\end{que}
\begin{que}
Let $Y$ be a Brieskorn homology 3-sphere with finite $E_8$-genus.
Then is $g_8(Y)=\overline{g_8}(Y)$ true?
\end{que}
We post a inequalities for bounding genus which are presumed by $\frac{11}{8}$-conjecture and Theorem~\ref{examnu}.
\begin{prob}
For positive integer $n$, we have
$$|\Sigma(8n-2,8n-1,16n-3)|,\ |\Sigma(8n-1,8n,16n-1)|\ge 3n$$
$$|\Sigma(8n-2,8n-1,32n^2-8n+1)|,\ |\Sigma(8n-1,8n,32n^2-1|\ge 3n.$$
\end{prob}
If one of this inequalities do not hold, then $\frac{11}{8}$-conjecture does not
hold.
\begin{que}
Do these homology 3-spheres above construct some $E_8$-boundings with $g_8=2n$?
\end{que}
\begin{que}
Let $a_k$ denote the 2nd homology class $k[f]-[s]$ in $E(n)$,
where $f$ is the general fiber and $s$ is a section.
Does there exists an upper bound of $k$ for
$a_k$ to be represented by an embedded $S^2$?
\end{que}
\begin{que}
For two homology 3-spheres with $\frak{ds}(X_i)<\infty$ ($i=1,2$),
Let denote $\tilde{\frak{ds}}(Y)=\epsilon(Y)\frak{ds}(Y)$.
Then when does the equality 
$$\tilde{\frak{ds}}(X_1)+\tilde{\frak{ds}}(X_2)=\tilde{\frak{ds}}(X_1\#X_2)$$
hold?
\end{que}
Although, in the case of $X_1=\Sigma(2,3,17)$ and $X_2=\Sigma(2,3,13)\#(-\Sigma(2,3,17))$,
the equality holds,
this equality seems unlikely, in general.
In order to satisfy this equality, some geometrically special condition would be occurred.

Finally, we post future's direction for this paper's topic.
\begin{prob}
Find more general constructions of positive (or negative) $E_8$-boundings for many homology 3-spheres.
\end{prob}

\noindent
Motoo Tange\\
University of Tsukuba, \\
Ibaraki 305-8502, Japan. \\
tange@math.tsukuba.ac.jp

\end{document}